\documentclass[12pt]{amsart}
\usepackage{a4wide,amssymb,amsthm,amscd,amsmath,mathtools,stmaryrd,url,fullpage,enumerate,color}
\usepackage[initials,alphabetic]{amsrefs}
\usepackage{enumitem}
\usepackage{bm}
\usepackage{tikz-cd}
\usepackage{stmaryrd}
\usepackage{stackengine}
\stackMath

\usepackage[all,cmtip]{xy}
\usepackage{xcolor}
\usepackage[mathscr]{euscript}
\usepackage[pdftex]{hyperref}
\definecolor{couleur_cite}{rgb}{0.05,.4,0.05}
\definecolor{couleur_link}{rgb}{0.05,0.05,0.4}
\hypersetup{
bookmarksopen,bookmarksnumbered,
colorlinks,
linkcolor=couleur_link,citecolor=couleur_cite,
}

\newcommand{\SO}{\textrm{SO}}

\newcommand{\C}{{\mathbb C}}
\newcommand{\N}{{\mathbb N}}
\newcommand{\R}{{\mathbb R}}
\newcommand{\Z}{{\mathbb Z}}
\newcommand{\Q}{{\mathbb Q}}

\newcommand{\bs}{\backslash}

\renewcommand{\Im}{\operatorname{Im}}
\renewcommand{\Re}{\operatorname{Re}}

\newcommand{\Ad}{\operatorname{Ad}}

\newcommand{\Hom}{\operatorname{Hom}}

\newcommand{\vol}{\operatorname{vol}}
\newcommand{\SL}{\operatorname{SL}}

\newcommand{\eps}{\epsilon}
\newcommand{\ad}{\operatorname{ad}}

\renewcommand{\Re}{\operatorname{Re}}
\renewcommand{\Im}{\operatorname{Im}}

\newcommand{\CmC}{\mathscr{C}}

\newcommand{\af}{\mathfrak{a}}
\newcommand{\gf}{\mathfrak{g}}

  \newcommand {\mf}{{\mathfrak m}}
  \newcommand {\kf}{{\mathfrak k}}

  \newcommand {\nf}{{\mathfrak n}}

   \newcommand {\pf}{{\mathfrak p}}

\renewcommand {\H}{{\mathscr H}}

 \newcommand {\cO}{{\mathscr O}}
 
 \newcommand {\cH}{{\mathscr H}}
 \newcommand {\cP}{{\mathscr P}}
 \newcommand {\cL}{{\mathscr L}}
 
 \newcommand {\cM}{{\mathscr M}}
\newcommand {\cT}{{\mathscr T}}

\newcommand {\cD}{{\mathscr D}}

\newcommand {\cE}{{\mathscr E}}
\newcommand {\cC}{{\mathscr C}}

\newtheorem{theorem}{Theorem}[section]
\newtheorem{lemma}[theorem]{Lemma}
\newtheorem{proposition}[theorem]{Proposition}

\newtheorem{definition}[theorem]{Definition}
\newtheorem{corollary}[theorem]{Corollary}

\theoremstyle{remark}
\newtheorem{remark}{Remark}[section]

\newcommand{\InjR}{\operatorname{InjRad}}

\def\Xint#1{\mathchoice
   {\XXint\displaystyle\textstyle{#1}}%
   {\XXint\textstyle\scriptstyle{#1}}%
   {\XXint\scriptstyle\scriptscriptstyle{#1}}%
   {\XXint\scriptscriptstyle\scriptscriptstyle{#1}}%
   \!\int}
\def\XXint#1#2#3{{\setbox0=\hbox{$#1{#2#3}{\int}$}
     \vcenter{\hbox{$#2#3$}}\kern-.5\wd0}}

\def\dashint{\Xint-}

\title[Quantum ergodicity in higher rank]{Quantum ergodicity for compact quotients of $\SL_d(\R)/\SO(d)$ in the Benjamini--Schramm limit}

\author{Farrell Brumley}
\address{LAGA - Institut Galil\'ee\\
99 avenue Jean Baptiste Cl\'ement\\
93430 Villetaneuse\\
France}
\email{brumley@math.univ-paris13.fr}
\thanks{The first author was supported by ANR grant 14-CE25}

\author{Jasmin Matz}
\address{Department of Mathematical Science\\
University of Copenhagen\\
Universitetsparken 5\\
2100 Copenhagen\\
Denmark}
\email{matz@math.ku.dk}

\begin{document}
\begin{abstract}
We study the limiting behavior of Maass forms on sequences of large volume compact quotients of $\SL_d(\R)/\SO(d)$, $d\ge 3$, whose spectral parameter stays in a fixed
window. We prove a form of Quantum Ergodicity in this level aspect
which extends results of Le Masson and Sahlsten to the higher rank case.
\end{abstract}

\maketitle
\setcounter{tocdepth}{1}
\tableofcontents

\section{Introduction}

Let $Y$ be a closed  Riemann manifold. Let $\mathscr{B}=\{\psi_i\}$ be an orthonormal basis of $L^2(Y)$ consisting of Laplacian eigenfunctions $\Delta\psi_i=\lambda_i\psi_i$. A subsequence $\{\psi_{i_j}\}$ of $\mathscr{B}$ is called \textit{quantum ergodic} if, for every degree 0 pseudo-differential operator $A$ on $Y$ with principal symbol $a\in C(S^*Y)$, we have
\[
\langle A\psi_{i_j},\psi_{i_j}\rangle_{L^2(Y)}\rightarrow \dashint_{S^*Y} ad\mu_L,\qquad \text{as}\quad j\rightarrow\infty.
\]
Here $\mu_L$ is the Liouville measure on the cosphere bundle $S^*Y$ and $\dashint$ indicates normalization by the volume. The quantum ergodicity theorem of \v{S}nirel\textprime man \cite{Sch}, Zelditch \cite{Zel}, and Colin de Verdi\`ere \cite{CdV} states that, if the geodesic flow on $S^*Y$ is ergodic, one may extract from $\mathscr{B}$ a density one quantum ergodic subsequence. More precisely, they show
\begin{equation}\label{spectral-QE}
\frac{1}{N(\lambda)}\sum_{\lambda_i\leq \lambda}\bigg|\langle A\psi_i,\psi_i\rangle_{L^2(Y)}- \dashint_{S^*Y} ad\mu_L\bigg|^2\rightarrow 0
\end{equation}
as $\lambda\rightarrow\infty$, where $N(\lambda)=|\{i: \lambda_i\leq\lambda\}|$.

In this paper we are concerned with a version of quantum ergodicity where, in contrast to the above semi-classical statement in which the manifold is fixed and the eigenvalue goes to infinity (the large frequency regime), we allow the manifold to vary while keeping the eigenvalue constrained to a fixed spectral window (the large spacial regime).

The most natural setting in which one can formulate such variation is in the Gromov--Hausdorff space of pointed locally compact spaces -- or, rather, its space of probability measures -- where one has a notion of convergence due to Benjamini and Schramm \cite{BS}. Here, one may consider a sequence of manifolds which converge, almost everywhere, to their common universal cover. 
Under certain auxiliary conditions, the question of Quantum Ergodicity for Benjamini--Schramm convergent sequences was recently settled for a large class of rank one symmetric spaces \cite{LS, ABL}.
The aim of the present paper is to address this question for higher rank locally symmetric spaces; for the most part we restrict ourselves to compact quotients associated with $\SL_d(\R)$.

\subsection{The setting of locally symmetric spaces}
As our setting will henceforth be that of locally symmetric spaces, we begin by commenting on some of their particular features in higher rank, and review what is known for their quantum ergodic properties in the large frequency regime. 

We let $S$ be a Riemannian globally symmetric space of non-compact type: non-positively curved and having no Euclidean de Rham local factor. We may write $S=G/K$ where $G$ is a connected semisimple Lie group with finite center, and $K$ is a maximal compact subgroup. The rank of $S$ is defined to be the dimension of a maximal flat subspace in $S$; equivalently, it is the dimension of a maximal split torus $A$ in $G$.

In rank one, $S$ is of constant negative curvature, and $A$ gives rise to the geodesic flow on the cosphere bundle of $S$, via the identification of the latter  with $G/M$, where $M$ is the centralizer of $A$ in $K$. It is important to note, however, that when the rank is strictly greater than one, such as for $\SL_d(\R)/\SO(d)$ when $d\geq 3$, the curvature of $S$ is non-constant, as indeed any geodesic triangle in a maximal flat will be Euclidean.

Now let $\Gamma\subset G$ be a uniform lattice in $G$ and form $Y=\Gamma\backslash S$, a compact locally symmetric space with geometry $S$. In higher rank, due to the presence of maximal flats, the geodesic flow is not ergodic on the cosphere bundle $\Gamma\backslash G/M$ of $Y$. For this reason, the quantum ergodic statement \eqref{spectral-QE} cannot hold, as it is known to be equivalent with the ergodicity of the geodesic flow \cite[Theorem 1]{Zel-Survey}.

Nevertheless, the higher rank split torus $A$ does act ergodically -- in fact it is mixing -- on $\Gamma\backslash G$, and this is enough for one to expect similar quantum ergodic phenomena to that of \v{S}nirel\textprime man's theorem, at least if one refines the notion of eigenfunction as follows.

Recall that a Maass form on $Y$ is a function $\psi\in L^2(Y)$ which is a joint eigenfunction of the algebra $\mathscr{D}(S)$ of left-$G$-invariant differential operators on $S$. Since $\Delta\in\mathscr{D}(S)$ a Maass form is again a Laplacian eigenfunction, but in higher rank   $L^2(Y)$ can be further diagonalized. A Maass form $\psi$ gives rise to an algebra homomorphism $\chi\in \Hom_{\C-{\rm alg}}(\mathscr{D}(S),\C)$ verifying $D\psi=\chi(D)\psi$ for all $D\in\mathscr{D}(S)$. The Harish-Chandra homomorphism allows one to realize $\chi$ as the Weyl group orbit of an element $\nu$ in $\mathfrak{a}^*_\C$, the complexification of the dual of the Lie algebra $\af$ of $A$. We call $\nu$ the \textit{spectral parameter} of $\psi$.

One can then formulate a natural extension of \eqref{spectral-QE} for Maass forms on $Y$ with growing spectral parameter. To the best of our knowledge, this form of Quantum Ergodicity has not yet been established, although recent work by Nelson and Venkatesh \cite{NV} on microlocal analysis and representation theory should shed light on this problem. 

\begin{remark}
There are of course harder conjectures that one can formulate here, such as Quantum \textit{Unique} Ergodicity \cite[Conjecture 1.2]{Lind1}. In the arithmetic setting, certain higher rank congruence manifolds have been shown \cite{SiVe1, SiVe2} to satisfy the Arithmetic Quantum Unique Ergodicity property, a generalization of the work (and techniques) of Lindenstrauss \cite{Lind}. We will not make further comments on this active line of research, and refer the reader to \cite[\S 1.3]{LS} for a discussion of a conjectural strengthening of Quantum Ergodicity for varying manifolds.
\end{remark}

\subsection{Our main result}
We now pass to the large spacial regime for higher rank locally symmetric spaces $Y=\Gamma\bs S$, and allow $\Gamma$ to vary along a non-conjugate sequence of torsion free uniform lattices with growing covolume. 

We furthermore require that the sequence $\Gamma$ be uniformly discrete, a property which we now recall. Let $d$ be the Riemannian distance on $S$. Then the local injectivity radius about a point $x\in Y$ is the quantity
\[
{\rm InjRad}_\Gamma(x)=\frac12\min\{ d(x,\gamma.x): 1\neq \gamma\in \Gamma\}.
\]
 The global injectivity radius ${\rm InjRad}(Y)$ is the infimum of ${\rm InjRad}_\Gamma(x)$ over all $x\in Y$; this is strictly positive since $Y$ is compact. We say that that a sequence of uniform torsion free lattices $\Gamma_n$ is \textit{uniformly discrete} if ${\rm InjRad}(Y_n)$ is bounded away from zero. A conjecture of Margulis \cite{Mar} states that this is automatic in higher rank; this is a weak form of the Lehmer conjecture on monic integral polynomials.

Our precise result is as follows. We refer to \S\ref{sec:gen-notation} for any unexplained notation.

\begin{theorem}\label{thm:QE}
Let $d\geq 3$. Let $\Gamma_n\subset\SL_d(\R)$ be a uniformly discrete sequence of torsion free cocompact lattices such that $\vol(Y_n)\rightarrow\infty$ as $n\rightarrow\infty$, where $Y_n=\Gamma_n\backslash\SL_d(\R)/{\rm SO}(d)$. Let $a_n$ be a sequence of uniformly bounded, measurable functions on $Y_n$.  

There is $\varrho>1$ such that for sufficiently regular $\nu\in i\af^*$ we have
\[
\frac{1}{N(B_0(\nu,\varrho),\Gamma_n)}\sum_{j:\,\nu_j^{(n)}\in B_0(\nu,\varrho)} \bigg|\langle a_n\psi_j^{(n)}, \psi_j^{(n)}\rangle_{L^2(Y_n)}-\dashint_{Y_n}a_n d\vol_{Y_n}\bigg|^2\rightarrow 0
\]
as $n\rightarrow\infty$, where $\dashint$ denotes the normalization of the integral by the volume of $Y_n$,
\[
B_0(\nu,\varrho)=\{\lambda \in i\mathfrak{a}^*: \|\lambda-\nu\|_2\leq\varrho\}
\]
is the ball of radius $\varrho$ in the unramified tempered spectrum, and 
\begin{equation}\label{defn:N(B,Gamma)}
N(B_0(\nu,\varrho),\Gamma_n)=|\{j: \nu_j^{(n)}\in B_0(\nu,\varrho)\}|.
\end{equation}
\end{theorem}

This extends to higher rank the rank one version of the same result in the papers \cite{LS} (for hyperbolic surfaces) and \cite{ABL} (for higher dimensional hyperbolic manifolds), which themselves built upon the breakthrough results of Anantharaman and Le Masson \cite{AL} for large regular graphs. All of these works imposed the following additional hypotheses:
\begin{enumerate}[label=(\roman{*})]
\item that the manifolds (or graphs) $Y_n$ converge to $S$ \textit{in the sense of Benjamini--Schramm}: for every $R>0$ we have
\[
\frac{\vol\left((Y_n)_{\leq R}\right)}{\vol(Y_n)}\rightarrow 0
\]
as $n\rightarrow\infty$, where the $R$-thin part of $Y_n$ is defined as
\begin{equation}\label{R-thin-part}
(Y_n)_{\leq R}=\{x\in Y_n: {\rm InjRad}_{\Gamma_n}(x)\leq R\};
\end{equation}
\item that the $Y_n$ have a uniform spectral gap. 
\end{enumerate}
The proof of Theorem \ref{thm:QE} also requires those properties, but they are automatic in higher rank; see \cite{Kazhdan} and \cite{AB+}*{\S 4}.

We note that in \cite{ABL} the large frequency regime for Laplacian eigenfunctions was interpreted as the Benjamini--Schramm convergence to Euclidean space via a rescaling of the Riemannian metric.

\begin{remark}
We point out two differences between the Quantum Ergodicity theorem of the recent preprint \cite{ABL} for rank one spaces and the higher rank version we present in Theorem \ref{thm:QE}. In contrast to our result, the authors of \cite{ABL} allow for two more general features, namely
\begin{enumerate}[label=(\alph*)]
\item\label{shrinking} they allow the spectral window to shrink with $n$, thereby isolating in the large $n$ limit a fixed tempered eigenvalue for the universal cover. This added flexibility was also present in \cite[Theorem 1.3]{AL};
\item\label{generalA} they take more general operators than scalar multiplication by functions $a_n$. This more advanced formulation was first put forward in \cite[Theorem 1.7]{AL}, using the pseudo-differential calculus for trees developed in \cite{L}.
\end{enumerate}
Our Theorem \ref{thm:QE} therefore more closely resembles the main results of \cite{BLL} and \cite{LS}, in which neither of these two features is present. By taking more elaborate test functions, we believe we can incorporate \ref{shrinking} into our set-up. By contrast, we have so far been unable to extend \ref{generalA} to this higher rank setting.
\end{remark}

\subsection{Comments on the proof}
To prove the theorem we first reduce the assertion to several intermediate statements as explained in Section \ref{sec:reduction}. This reduction roughly follows along the lines of \cite{LS} and the subsequent paper \cite{ABL}. To establish these intermediate results, however, we need several new ideas to deal with a number of issues that arise only in rank 2 or higher:
\begin{itemize}
\item The first main reduction step involves the use of a normalized averaging operator on $S=G/K$ (a kind of wave propagation) with expanding support $C_t$, $t\rightarrow\infty$. At a later critical point we need to estimate from above the volume of intersections $C_t\cap g C_t$ with $g\in G$. 
In \cite{LS, ABL} they work with $C_t$ being the metric ball $B_t$ in $S$ of radius $t$, and exploit the fact that $S$ is a ${\rm CAT}(-1)$ space in rank $1$.
In higher rank, $S$ is only a ${\rm CAT}(0)$ space, and working with intersection of metric balls becomes problematic. We therefore need to define new types of $C_t$ that look more `polytopal'  and are easier to work with in higher rank; see \eqref{Et-polar} and  Section \ref{subsec:intersection}.

\item We need to establish a suitable lower bound for certain averages of spherical functions. In rank 1, this can be dealt with in a  relatively straightforward manner, as the elementary spherical functions are basically linear combinations of trigonometric functions in one variable. In higher rank, we need to deal with linear combinations of exponential functions in several variables which makes the analysis much more delicate; see Section \ref{sec:spectral}. The techniques of that section might also be of interest in other contexts.

\end{itemize}

Additionally, we need to establish a type of local Weyl law/limit multiplicity for the $Y_n$ that gives a lower count the number of eigenvalues locally around sufficiently regular points in the spectrum of $Y_n$. We only need that sharp lower bound in the level aspect as stated in Proposition \ref{thm:spec-var}, but along the way prove a stronger version that also yields the right order in the spectral parameter, see \eqref{more-precise}. Such a result is also necessary for the rank $1$ situation, but involves a much more careful analysis of the non-tempered spectrum in higher rank.

\section{Outline of proof and reduction steps}\label{sec:reduction}

In this section, we shall reduce the proof of Theorem \ref{thm:QE} to that of two auxiliary estimates: one spectral, one geometric. Let $G=\SL_d(\R)$ and $K=\SO(d)$. As a preliminary step, we note that by replacing $a_n$ by $a_n-\dashint_{Y_n} a_n$ we may suppose that the measurable, right-$K$-invariant functions $a_n$ on $\Gamma_n\backslash G$ satisfy
\begin{equation}\label{function-hypoth}
\int_{\Gamma_n\backslash G} a_n=0  \qquad \text{and}\qquad  \|a_n\|_\infty\leq 1.
\end{equation}
Under that assumption it will then be enough to prove
\begin{equation}\label{lem:reduction}
\frac{1}{N(B_0(\nu,\varrho),\Gamma_n)}\sum_{j:\,\nu_j^{(n)}\in B_0(\nu,\varrho)} |\langle a_n\psi_j^{(n)}, \psi_j^{(n)}\rangle_{L^2(Y_n)}|^2\rightarrow 0
\end{equation}
as $n\rightarrow \infty$. The rest of the paper is devoted to establishing \eqref{lem:reduction}. 

\subsection{Spectral estimate}
We shall first need to control the spectral counting function defined in \eqref{defn:N(B,Gamma)}. This is provided in the following result, a sharp spectral lower bound, to be proved in Section \ref{sec:Weyl}. See \S\ref{sec:gen-notation} for our conventions on the dependence on the implied constants on the sufficiently regular assumption.

\begin{proposition}\label{thm:spec-var}
Let $G$ be a connected non-compact simple Lie group with finite center, and $K$ a maximal compact subgroup. Let $S=G/K$ be the associated irreducible Riemannian globally symmetric space of non-compact type. Let $\Gamma_n$ be a sequence of uniformly discrete, torsion free, cocompact lattices in $G$ such that $Y_n=\Gamma_n\backslash S$ converges, in the sense of Benjamini--Schramm, to $S$. There is $\varrho>1$ such that, for any sufficiently regular $\nu\in i\af^*$, 
\[
\vol(\Gamma_n\backslash G)\ll N(B(\nu,\varrho),\Gamma_n),
\]
where $B(\nu,\varrho)=\{\lambda\in\af_\C^*\mid \|\lambda - \nu\|_2<\varrho\}$, and $N(B(\nu,\varrho),\Gamma_n)=|\{j\mid \nu_j^{(n)}\in B(\nu,\varrho)\}|$ is the counting function.
\end{proposition}

The proof of Proposition \ref{thm:spec-var} in fact yields more information than what we have recorded here. See \eqref{more-precise} for the precise estimate. In particular, for sufficiently regular $\nu\in i\af^*$ we show the stronger uniform lower bound $\vol(\Gamma_n\backslash G)\tilde\beta(\nu)\ll N(B(\nu,\varrho),\Gamma_n)$.

We note that Proposition \ref{thm:spec-var} is stated for general symmetric spaces $S$. By contrast, we have restricted the setting of the next two results, Theorems \ref{thm:spec} and \ref{thm:GeomSide} below, to the symmetric space $S=\SL_d(\R)/\SO(d)$.

\subsection{The averaging subset}\label{sec:Et}

Following \cite{BLL} and \cite{LS}, the main idea behind \eqref{lem:reduction} is the strategic use of a self-adjoint normalized averaging operator which, on one hand, acts non-decreasingly on the spectral average (Theorem \ref{thm:spec}) and on the other hand has small Hilbert--Schmidt norm (Theorem \ref{thm:GeomSide}). This idea can be traced back to the work of Brooks, Le Masson, and Lindenstrauss \cite{BLL}, which presents an alternative proof of \cite[Theorem 1.3]{AL}, using discrete normalized averaging operators.

Below we give a definition of an exhaustive sequence of sets $E_t\subset \SL_d(\R)$ which we shall use in our averaging operators. The difficulty in higher rank is finding a set $E_t$ which satisfies simultaneously the desired spectral and geometric properties (in rank 1, the obvious choice of a Riemannian metric ball is shown to work in \cite{LS} and \cite{ABL}). 

We begin by introducing some notation. Let $\gf$ be the trace zero matrices in $M_d(\R)$. Let 
\[
\af=\big\{X={\rm diag}(X_1,\ldots ,X_d)\in\gf\big\}
\]
be the standard Cartan subalgebra of diagonal matrices. Let $W=N_K(\af)/Z_K(\af)$ be the Weyl group of $(\af,\gf)$; the action of $W\simeq S_d$ on $\af$ is by permutation of the coordinates.  Let $\af^+$ be the standard positive Weyl chamber in $\af$ given by $X_1>X_2>\cdots >X_d$. For a vector $X\in \af$ we will write
\[
\|X\|_\infty=\max\{|X_1|,\ldots, |X_d|\}.
\]
The $W$-invariant norm $\|\cdot \|_\infty$ on $\af$ induces a sub-additive bi-$K$-invariant norm on $G=\SL_d(\R)$ via the Cartan decomposition $G=K\exp (\overline{\af^+})K$. Namely, if $g\in G$ with $g\in K e^X K$, then we write $|g| = \|X\|_\infty$. We have $|g|\geq 0$, $|g^{-1}|=|g|$, and $|g_1g_2|\leq |g_1|+|g_2|$.

With this notation, we put
\begin{equation}\label{def-Et}
E_t=\{g\in G: |g| \leq t\}.
\end{equation}
This is the averaging set we shall use for our wave propagator. One may view $E_t$ as a radially invariant subset of the symmetric space $G/K$. See \S\ref{Et-remarks} for more commentary on the nature of $E_t$. 

\subsection{The two main results} We now define 
\[
k_t=\frac{1}{\sqrt{m_G(E_t)}}\mathbf{1}_{E_t}.
\]
Let $\rho_{\Gamma\backslash G}$ denote the right-regular representation of $G$ on $L^2(\Gamma\backslash G)$, and consider the wave-propagation operator on $L^2(\Gamma\backslash G)$ given by
\begin{equation}\label{defn-propogator}
U_t=\rho_{\Gamma\backslash G}(k_t).
\end{equation}
From $E_t^{-1}=E_t$ it follows that $k_t$ is a self-adjoint operator. For a measurable right-$K$-invariant function $a$ on $\Gamma\backslash G$ satisfying \eqref{function-hypoth}, and for $\tau>0$, we consider the time average
\begin{equation}\label{defn-Abar}
\mathbf{A}(\tau)=\frac{1}{\tau}\int_0^\tau U_taU_t dt.
\end{equation}
 
In Section \ref{sec:spectral} we shall prove the following spectral estimate.
\begin{theorem}[Spectral estimate]\label{thm:spec}
Let $S=\SL_d(\R)/{\rm SO}(d)$. Let $\Gamma$ a cocompact lattice in $\SL_d(\R)$, and put $Y=\Gamma\bs S$. Let $\Omega\subseteq i\af^*$ be compact. There exist constants $c,\tau_0>0$, depending on $\Omega$, such that, for $\tau\ge\tau_0$, we have
\[
\sum_{j:\, \nu_j\in\Omega} |\langle a\psi_j, \psi_j\rangle_{L^2(Y)}|^2\leq c \sum_{j:\,  \nu_j\in\Omega} \bigg|\langle \mathbf{A}(\tau)\psi_j , \psi_j\rangle_{L^2(Y)}\bigg|^2.
\]
\end{theorem}

In Section \ref{sec:geom}, we shall use the Nevo mean ergodic theorem \cite{Nevo} (see also \cite[Theorem 4.1]{GV}) to prove the following geometric estimate. Recall that the Hilbert--Schmidt norm of a bounded operator $\mathbf{A}$ on a separable Hilbert space $H$ is given by 
\[
\|\mathbf{A}\|_{\rm HS}^2=\sum_i |\langle \mathbf{A} e_i,e_i\rangle|^2,
\]
where $\{e_i\}$ is any orthonormal basis of $H$. 

\begin{theorem}[Geometric estimate]\label{thm:GeomSide}
Let $S=\SL_d(\R)/{\rm SO}(d)$. Let $\Gamma$ be a cocompact torsion free lattice in $\SL_d(\R)$, and put $Y=\Gamma\bs S$. Let $a$ be a measurable function on $L^2(\Gamma\backslash S)$. There are constants $b, c_1,c_2>0$, depending only on $d$, such that, for $\tau>0$,
\[
\big\|\mathbf{A}(\tau)\big\|_{\rm{HS}}^2\ll_b \frac{\|a\|_2^2}{\tau}+ \frac{e^{c_1\tau}}{{\rm InjRad}(Y)^{\dim S}}\vol((\Gamma\backslash G)_{\leq c_2(2\tau+b)})\|a\|_\infty^2.
\]
\end{theorem}

\subsection{Reduction to two main results}
We now deduce the main estimate \eqref{lem:reduction} from the above results. 

We first note  that if $\nu\in i\af^*$ is sufficiently regular in the sense that $|\langle\nu,\alpha^\vee\rangle|\ge C_\varrho$ for all simple roots $\alpha$ with $C_\varrho>0$ a sufficiently large constant depending on $\varrho$, then $B_0(\nu,\varrho)=B(\nu,\varrho)$ because of \eqref{eq:unitary:rep}. 
By Proposition \ref{thm:spec-var} it suffices to prove an upper bound for
\[
\frac{1}{\vol (\Gamma_n\bs G)}\sum_{j:\, \nu_j^{(n)}\in B(\nu,\varrho)} |\langle a_n\psi_j^{(n)}, \psi_j^{(n)}\rangle_{L^2(Y_n)}|^2.
\]
From Theorem \ref{thm:spec} it follows that, for $\tau\geq \tau_0$, 
\[
\frac{1}{\vol (\Gamma_n\bs G)}\sum_{j:\, \nu_j^{(n)}\in B(\nu,\varrho)} |\langle a_n\psi_j^{(n)}, \psi_j^{(n)}\rangle_{L^2(Y_n)}|^2 \ll \frac{1}{\vol (\Gamma_n\bs G)}\big\| \mathbf{A}^{(n)}(\tau)\big\|_{\textrm{HS}}^2.
\]
We now apply Theorem \ref{thm:GeomSide} to the right-hand side. To bound the first term, we use the estimate $\|a_n\|_2^2\leq \|a_n\|_\infty^2 \vol(\Gamma_n\backslash G)$ along with the uniform boundedness of $a_n$. For the second term we insert the bound ${\rm InjRad}(Y_n)\gg 1$ coming from the uniformly discrete hypothesis in the statement of Theorem \ref{thm:QE}. We obtain
\begin{equation}\label{bound-2-terms}
\frac{1}{\vol (\Gamma_n\bs G)}\big\|\mathbf{A}^{(n)}(\tau_n)\big\|_{\rm{HS}}^2\ll \frac{1}{\tau_n}+ e^{c_1\tau_n}\frac{\vol\left((\Gamma_n\backslash G)_{\leq c_2(2\tau_n+b)}\right)}{\vol (\Gamma_n\bs G)}.
\end{equation}
We know from \cite{AB+}*{\S 4} that for any sequence $\Gamma_n$ for which $\vol(\Gamma_n\backslash G)\rightarrow\infty$ the space $\Gamma_n\backslash G/K$ converges in the sense of Benjamini--Schramm to $G/K=\SL_d(\R)/\SO(d)$ (recall our assumption that $d\geq 3$). We deduce that there is a sequence $R_n\rightarrow\infty$ such that
\begin{equation}\label{little-oh}
\frac{\vol\left((\Gamma_n\backslash G)_{\leq R_n}\right)}{\vol(\Gamma_n\backslash G)}\rightarrow 0
\end{equation}
as $n\rightarrow\infty$. Given this sequence $R_n$ we let $r_n>0$ be a sequence such that
\begin{enumerate}
\item[(i)] $r_n\rightarrow \infty$ as $n\rightarrow\infty$,
\medskip
\item[(ii)] $e^{c_1 r_n}\frac{\vol\left((\Gamma_n\backslash G)_{\leq c_2(2r_n+b)}\right)}{\vol(\Gamma_n\backslash G)}\rightarrow 0$ as $n\rightarrow\infty$.
\medskip
\end{enumerate}
Taking $\tau_n=r_n$, both terms in \eqref{bound-2-terms} go to $0$ with $n$, establishing \eqref{lem:reduction}.

\begin{remark}
One could deduce an effective rate of convergence in Theorem \ref{thm:QE} for congruence subgroups $\Gamma_n$ of $\SL_d(\Z)$ by inserting, in place of the $o(1)$ result in \eqref{little-oh}, the explicit bounds on the $R$-thin part for such sequences proved in \cite[Theorem 5.2]{AB+}.
\end{remark}

\subsection{Remarks on the averaging subset}\label{Et-remarks}

We now make several remarks about $E_t$:

\medskip

\noindent (i) In rank one, the set $E_t$ recovers the metric ball centered at $i\in\mathbb{H}$ of radius $2t$ for the hyperbolic metric $(dx^2+dy^2)/y^2$ on the upper half-plane $\mathbb{H}=\SL_2(\R)/\SO(2)$. 

\medskip

\noindent (ii) In higher rank, the set $E_t$ differs substantially from a metric ball centered at the identity for the metric induced by the trace form ${\rm Tr}(X^\top Y)$ (a constant multiple of the Killing form) on $\gf$. 

To see this, let us first compare a geodesic ball and the set $E_t$ by means of the Cartan decomposition of $G=\SL_d(\R)$. Let $\mathscr{B}=\{X\in\af: \|X\|_2\leq 1\}$, where $\|X\|_2^2={\rm tr}(X^2)=X_1^2+\cdots +X_d^2$. The geodesic ball $B_t$ of radius $t>0$ centered at the origin is given by
\begin{equation}\label{Bt-polar}
B_t=K\exp (\mathscr{B}_t^+) K,
\end{equation}
where $\mathscr{B}_t^+=t\mathscr{B}\cap\overline{\af^+}$; see \cite[Proposition 4.2]{BM}. On the other hand, let
\begin{equation}\label{def:cP}
\cP=\{X\in\af\mid  \|X\|_\infty\le 1\}.
\end{equation}
Note that $\cP$ is a bounded convex polytope in $\af$, being the intersection of the half planes $\pm X_i\leq 1$. It is clear that
\begin{equation}\label{Et-polar}
E_t=K\exp(\cP_t^+)K,
\end{equation} 
where $\cP^+_t=t\cP\cap \overline{\af^+}$. One sees from \eqref{Bt-polar} and \eqref{Et-polar} that an element in $B_t$ or $E_t$ has Cartan radial component contained, respectively, to an expanding ball or a dilated $W$-invariant polytope in $\af$.

\medskip

\noindent (iii) Related to the above description is the difference in the volume asymptotics between $B_t$ and $E_t$. Both are expressed, naturally, with respect to the half-sum of the positive roots $\rho\in\af^*$, as that quantity governs the Jacobian factor in the Cartan decomposition of the Haar measure on $G$; see \S\ref{subsection:ideas}. 

For the geodesic ball, a result of Knieper \cite{Knieper}, valid more generally for irreducible Riemannian symmetric spaces $S=G/K$ of non-compact type, states that
\begin{equation}\label{Knieper-bound}
m_G(B_t)\asymp  t^{(r-1)/2}e^{2t\|\rho\|_2},
\end{equation}
where $r$ is the rank of $S$. On the other hand, we show in Corollary \ref{rem:volume} that (for $G=\SL_d(\R)$)
\[
m_G(E_t)\sim c_2 e^{2t\langle \rho, X^0\rangle}\qquad (t\rightarrow\infty),
 \]
where $X^0\in \cP^+$ is the unique point satisfying
 \begin{equation}\label{defnX0}
\langle \rho, X^0\rangle=\max_{X\in\cP^+}\langle\rho,X\rangle.
\end{equation}
One can compute that
\begin{equation}\label{rho-norm}
\|\rho\|_2=\big\langle\rho,\frac{\rho}{\|\rho\|_2}\big\rangle=\sqrt{d(d^2-1)/12},
\end{equation}
whereas, using Lemma \ref{lem:X0:explicit}, we have
\[
\big\langle\rho,\frac{X^0}{\|X^0\|_2}\big\rangle=\begin{cases} d^{3/2}/4&,\; d\; {\rm even}\\
(d+1)(d-1)^{1/2}/4&,\; d\; {\rm odd}.
\end{cases}
\]
We have normalized by $\|X^0\|_2$ so that $B_t$ is the smallest geodesic ball containing $E_{t/\|X^0\|_2}$. We see that the exponential volume growth of $B_t$ in $\SL_d(\R)/\SO(d)$ is of order $2t \sqrt{d^3/12}$, whereas that of $E_{t/\|X^0\|_2}$ is of order $2t\sqrt{d^3/16}$, which is a factor of $\sqrt{3}/2$ smaller.
\medskip

\noindent (iv) Another norm that one often considers in the context of $G=\SL_d(\R)$ is the restriction to $\SL_d(\R)\subset M_d(\R)$ of the \textit{Frobenius norm}, defined on $M_d(\R)$ as $\|g\|^2={\rm tr}(g^\top g)$, where $g^\top$ is the transpose of $g$. Note that when $\|g^{-1}\|\neq \|g\|$ in general. Since invariance under inversion is important for the self-adjointness of the propagator $U_t$ from \eqref{defn-propogator}, we define the \textit{Frobenius ball} as 
\begin{equation}\label{def:Frob-ball}
\bm{E}_t=\left\{g\in G: \max\{\|g\|,\|g^{-1}\|\}\leq e^t\right\}.
\end{equation}
From the point of view of their large scale geometry, the Frobenius balls are similar to the sets $E_t$ from \eqref{def-Et}. Indeed in the  proof of Proposition \ref{prop:upper:bound:intersection} we show that $m_G(E_t)\asymp m_G(\bm{E}_t)$. Moreover, the very statement Proposition \ref{prop:upper:bound:intersection}, which is a key component Theorem \ref{thm:GeomSide}, holds equally well for $\bm{E}_t$. On the other hand, because they are not defined by their radial Cartan component, the Frobenius balls are not amenable to our proof of Theorem \ref{thm:spec}.

\section{Weyl type law}\label{sec:Weyl}
Our aim in this section is to prove Proposition \ref{thm:spec-var}. Throughout this section only, we let $G$ denote a connected non-compact simple Lie group with finite center, $K$ a maximal compact subgroup. 

\subsection{Notation}\label{sec:gen-notation}
The notation we introduce here will be consistent with that already introduced for the particular case of $G=\SL_d(\R)$ and $K=\SO(d)$.

Let $\theta$ be the Cartan involution on $G$ for which $K=G^\theta$. Let $\Theta$ be its differential and let $\gf=\pf\oplus\kf$ be the Cartan decomposition of the Lie algebra $\gf$ of $G$ into the $\pm 1$ eigenspaces of $\Theta$. We may identify $\kf$ with the Lie algebra of $K$. 

Let $\kappa(X,Y)={\rm tr}(\ad X\ad Y)$ denote the Killing form on $\gf$. Then $-\kappa(X,\Theta Y)$ defines an $\Ad$-invariant inner product on $\gf$, which induces a left-invariant Haar measure on $G$, denoted $dg$ or $m_G$.

The Killing form defines an $\Ad_K$-invariant inner product $\langle \cdot,\cdot\rangle$ on $\pf$, and in particular on $\af$, the maximal abelian subalgebra of $\pf$. Let $W=N_K(\af)/Z_K(\af)$ be the Weyl group. Let $\| \cdot\|_2$ denote the $W$-invariant norm\footnote{When $G=\SL_d(\R)$, we introduced in \S\ref{Et-remarks} (ii) the norm induced by the trace form on $\gf$, denoted there by the same symbol $\|\cdot \|_2$. The Killing form for $\SL_d(\R)$, in fact, differs from the trace form by a constant factor of $2d$. When we return to the specific situation of $\SL_d(\R)$ in later sections, we shall always take the norm $\|\cdot \|_2$ on $\af$ to mean $\|X\|_2^2={\rm tr}(X^2)$.} on $\af$ induced by $\langle \cdot,\cdot\rangle$. 

Let $\Phi\subset \af^*$ denote the system of roots for the adjoint action of $\af$ on $\gf$. For $\alpha\in\Phi$ let $\gf_\alpha$ denote the corresponding root space. Let $\mf\subset\gf$ be the centralizer of $\af$ in $\gf$. Then we have the root space decomposition
\[
\gf=\mf\oplus\af\oplus\bigoplus_{\alpha\in\Phi}\gf_\alpha.
\]
Choose a positive system of roots $\Phi^+$ in $\Phi$. Let $\Delta\subseteq \Phi^+$ be the set of simple roots. 

Let $\af^*={\rm Hom}(\af,\R)$ be the dual vector space of $\af$. We may identify $\af$ with $\af^*$ via the Killing form. We again denote by $\|\cdot \|_2$ the induced norm on $\af^*$ and extend it to a complex bilinear form on the complexification $\af^*_\C=\af^*\otimes_\R\C$. 
We call $\nu\in i\af^*$ \textit{regular} if $\langle \alpha,\nu\rangle\neq 0$ for all $\alpha\in\Delta$, and for $c>0$ we call $\nu\in i\af^*$ \emph{$c$-regular} if $|\langle \alpha,\nu\rangle|\ge c$ for all $\alpha\in\Phi^+$. An element $\nu\in i\af^*$ is \emph{sufficiently regular} if there exists a sufficiently regular $c>0$ (which may vary at each occurrence) such that $\nu$ is $c$-regular. Unless otherwise noted, implied constants for statements valid for sufficiently regular $\nu$ can depend on the value of $c$.

Let $\nf=\sum_{\alpha\in\Phi^+} \gf_\alpha$ and write $N=\exp\nf$. Then we have the Iwasawa decomposition $G=NAK$, where $A=\exp\af$, which gives rise to the Iwasawa projection
\[
H_0: G\rightarrow \af, \qquad g=nak\mapsto \log a,
\]
along the $A$ component. 

We may decompose the Riemannian Haar measure $dg$ on $G$ according to the Iwasawa decomposition. We let $da$ denote the Haar measure on $A$ obtained by pushing forward the Lebesgue measure on $\af$ my means of the exponential. We let $dk$ denote the probability Haar measure $dk$ on $K$. We normalize the left-invariant Haar measure $du$ on $N$ as in \cite[\S 3.1]{DKV}. As usual we put
\[
\rho=\frac12{\rm tr}(\ad(\af)|_\nf)=\frac12\sum_{\alpha\in\Phi^+}(\dim\gf_\alpha) \alpha.
\]
Then one has \cite[Proposition 2.4.10]{GaVa}
\[
dg= c_Ie^{2\rho (H_0(g))}du da dk \qquad (c_I= 2^{-(1/2)\dim N}).
\]
Since $A$ normalizes $N$ and $\det \Ad(a)|_\nf=e^{2\rho (H_0(a))}$,
\begin{equation}\label{Iwasawa-measure}
dg=c_Ida du dk
\end{equation}
in the decomposition $G=ANK$.

Recall the Cartan decomposition $G=K\exp (\overline{\af^+})K$, where $\overline{\af^+}$ denotes the closure of the (open) positive Weyl chamber $\af^+=\{X\in\af: \alpha(X)>0\;\forall\,\alpha\in \Phi^+\}$. We have the following decomposition of the Riemannian Haar measure $dg$ into the Cartan decomposition
\begin{equation}\label{Cartan-measure-decomp}
dg=c_C J(X) dk_1 da dk_2 \qquad (c_C=2^{-\dim N}\vol (K)\vol(K/M)).
\end{equation}
Here, $\vol(K)$ is the Riemannian volume on $K$, for the measure induced by the inner product $-\kappa(X,Y)$ on $\mathfrak{k}$, and similarly with $\vol(K/M)$. Finally, $J(X)$ is the Jacobian factor, given by
\begin{equation}\label{first-J(X)}
 J(X)=\prod_{\alpha\in\Phi^+} \sinh(\alpha(X))^{\dim\gf_\alpha}\qquad (X\in\overline{\af^+});
\end{equation}
see \cite[Proposition 2.4.11]{GaVa}.

Let $P_0$ be the normalizer of $\mf\oplus\af\oplus\nf$ in $G$. Then $P_0$ has Langlands decomposition $P_0=MAN$, where $M$ is the centralizer of $A$ in $G$. For $\lambda\in\af^*_\C$ let $\pi_\lambda$ denote the unique unramified irreducible subquotient of ${\rm Ind}_{P_0}^G(1\otimes e^\lambda\otimes 1)$. Let
\[
\af^*_{\rm un}=\{\lambda\in\af^*_\C: \pi_\lambda \;\text{unitarizable}\}
\]
be the spherical unitary dual of $G$. Furthermore, let
\[
\af^*_{\rm hm}=\bigcup_{w\in W}\{\lambda\in\af^*_\C: w\lambda=-\bar\lambda\}
\]
denote the spherical hermitian dual of $G$ (see \cite[\S 3.3]{LM}). Let $i\af$ be the subspace of $\af_\C^*$ consisting of $\lambda$ taking on purely imaginary values. We may write $\lambda\in\af^*_\C$ uniquely as $\lambda={\rm Re}\,\lambda+{\rm Im}\,\lambda \in \af^*\oplus i\af^*$. We have
\begin{equation}\label{eq:unitary:rep}
i\af^*\subset \af^*_{\rm un}\subset \af^*_{\rm hm}\cap\{\lambda\in\af^*_\C: \|{\rm Re}\,\lambda\|_2\leq \|\rho\|_2\}.
\end{equation}
For $\mu\in i\af^*$ and $r>0$ we let 
\[
B_0(\mu,r)=\{\lambda\in i\af^*: \|\lambda-\mu\|_2\leq r\}
\]
denote the ball of radius $r$ about $\mu$ in the tempered spectrum $i\af^*$. When $r=1$ we write $B_0(\mu)$ for $B_0(\mu,1)$.
We also write 
\[
 B(\mu,r)=\{\lambda\in \af^*_\C: \|\lambda-\mu\|_2\leq r\}
\]
for the ball of radius $r$ around $\mu$ in $\af^*_\C$.

\subsection{Plancherel density and the $\mathbf{c}$-function}
We denote by $\beta$ the density function for the Plancherel measure on the spherical unitary dual, which can be identified with $\af^*_{\rm un}/W$. Then $\beta$ is a $W$-invariant function supported on $i\af^*$ that can be described as a product of $\Gamma$-functions \cite[Ch. IV, \S6]{Helgason}. Following \cite{LM} we put
\[
\tilde\beta(t,\lambda)=\prod_{\alpha\in\Phi^+}(t+|\langle \lambda,\alpha^\vee\rangle |)^{\dim\gf_\alpha}, \qquad \lambda\in i\af^*, \; t\geq 1,
\]
and $\tilde\beta(\lambda)=\tilde\beta(1,\lambda)$. From \cite[Ch. IV, Theorem 6.14]{Helgason} and standard estimates for the $\Gamma$-function it follows that $\beta(\lambda)\ll \tilde\beta(\lambda)$ for all $\lambda\in i\af^*$.
Since 
\[
 1+|\langle t\lambda +\mu,\alpha^\vee\rangle |
 \le 1+ t |\langle \lambda ,\alpha^\vee\rangle | + |\langle\mu,\alpha^\vee\rangle |
 \le t (1+ \|\lambda\|_2) +|\langle\mu,\alpha^\vee\rangle |
  \le (1+ \|\lambda\|_2)(t +|\langle\mu,\alpha^\vee\rangle |)
\]
for all $\lambda,\mu\in i\af^*$ and $t\geq 1$, we get
\begin{equation}\label{beta-tilde-bound}
\tilde\beta(t\lambda+\mu)\ll (1+\|\lambda\|_2)^{|\Phi^+|} \tilde\beta(t,\mu)
\end{equation}
for all such $\lambda,\mu$, and $t$.

The Harish-Chandra $\mathbf{c}$-function $\mathbf{c}:\af^*_\C\longrightarrow \C$ for $G$ asymptotically describes the behavior of the elementary spherical functions $\phi_\lambda(e^H)$ of $G$ as the group parameter $H$ grows. The quantity $\mathbf{c}(\lambda)$ depends only on the root system of $G$, and can be explicitly computed as described in \cite[Chapter IV, Theorem 6.14]{Helgason}. Up to normalization, the Plancherel density $\beta(\lambda)$ equals $|\mathbf{c}(\lambda)|^{-2}$ for $\lambda\in i\af^*$.

\subsection{Test functions}
For $\lambda\in\af^*_\C$ let
  \[
   \varphi_\lambda(g)=\int_K e^{\langle \lambda + \rho, H_0(kg)\rangle} \, dk
  \]
denote the spherical function on $G$ with spectral parameter $\lambda$. Here $H_0:G\longrightarrow \af$ is the Iwasawa projection. Let $C^\infty_c(G\sslash  K)$ be the space of complex-valued bi-$K$-invariant functions on $G$. The Harish-Chandra transform of a function $k\in C^\infty_c(G\sslash  K)$ is defined to be
\[
\hat{k}(\nu)=\int_G \varphi_\lambda(g) k(g)dg,
\]

For $t>0$ let $\mathscr{B}_t$ denote the ball of radius $t$ centered at $0$ in $\af$ with respect to the usual Euclidean norm $\|\cdot\|_2$. Let $G_{\leq t}=K\exp \mathscr{B}_t\, K$. Let $C^\infty_c(G\sslash  K)_{\leq t}$ denote the space of smooth compactly supported bi-$K$-invariant functions on $G$, supported on $G_{\leq t}$. Let $\mathscr{PW}(\af^*_\C)_t$ denote the class of Paley--Wiener functions of exponential type $t$. The Paley--Wiener theorem with supports (see \cite[Theorem 3.5]{Gangolli}) states that the Harish-Chandra transform  is a topological isomorphism $C^\infty_c(G\!\sslash \! K)_{\leq t}\xrightarrow{\sim} \mathscr{PW}(\af^*_\C)_t^W$ of Fr\'echet spaces. The inverse map sends $h\in \mathscr{PW}(\af^*_\C)_t^W$ to 
\[
\frac{1}{|W|}\int_{i\af^*}h(\mu)\varphi_\mu(g)\beta(\mu)d\mu.
\]
In particular, for $k\in C_c^\infty(G\sslash  K)$ we have
\begin{equation}\label{Plancherel-inversion}
k(e)=\frac{1}{|W|}\int_{i\af^*}\hat{k}(\mu)\beta(\mu)d\mu.
\end{equation}

We shall need a Paley--Wiener function concentrating about $\nu$ and verifying certain positivity properties.

\begin{lemma}\label{lemma:test-fn}\cite[\S 4]{BrMa}
For $t\geq 1$ and $\nu\in i\af^*$ there is $k_{\nu,t}\in C^\infty_c(G\!\!\sslash \!\! K)_{\leq 1/t}$ whose Harish-Chandra transform $h_{\nu,t}\in\mathscr{PW}(\af^*_\C)_t^W$ satisfies 
\begin{enumerate}
\item\label{1} $h_{\nu,t}$ is real and non-negative on $\af^*_{\rm hm}$;
\item\label{rapid-decay} for all $\lambda\in\af^*_{\rm un}$ we have
\[
h_{\nu,t}(\lambda)\ll_A (1+\|({\rm Im}\,\lambda-\nu)/t\|_2)^{-A};
\]
\item\label{size-h} there are constants $0<c_1,c_2<1$ such that for $\lambda\in\af_{\rm un}^*$ with $\|{\rm Im}\,\lambda-\nu\|_2\leq c_1t$ we have $c_2\leq h_{\nu,t}(\lambda)\leq 2$;
\item\label{size-k} $k_{\nu,t}(g)\ll\tilde\beta (\nu)t^r(1+\|\nu\|_2 d(g,K))^{-1/2}$, where $r=\dim\af$.
\end{enumerate}
\end{lemma}

\subsection{Proof of Proposition \ref{thm:spec-var}.}
Recall that for $\nu\in i\af^*$ and $c>0$, we denote by  $B(\nu, c)$ the ball $\{\lambda\in \af_\C^*\mid \|\lambda-\nu\|_2\le c\}$ in $\af_\C^*$. Moreover, for  any set $\Omega\subseteq \af^*_\C$,  we write $N(\Omega,\Gamma_n)=|\{j\mid \nu_j^{(n)}\in \Omega\}|$.

We shall in fact prove a stronger statement than Proposition \ref{thm:spec-var}. We shall show that there exist $0<\varrho_1<1<\varrho_2$ such that the following holds. Fix $\epsilon>0$, and let $\nu\in i\af^*$ be $\varepsilon$-regular.
Then
\begin{equation}\label{more-precise}
N(B(\nu,\varrho_1),\Gamma_n)\ll \vol(\Gamma_n\bs G)\tilde\beta(\nu)\ll_\varepsilon N(B(\nu,\varrho_2),\Gamma_n).
\end{equation}
Proposition \ref{thm:spec-var} then follows from the second estimate above and the fact that $\tilde\beta(\nu)\geq 1$ for all $\nu\in i\af^*$.
 
In the notation of Lemma \ref{lemma:test-fn}, we write $h_\nu=h_{\nu,1}$ and $k_\nu=k_{\nu,1}$. Then
\[
\sum_{\lambda\in \Lambda_n}h_\nu(\lambda) 
= \vol(\Gamma_n\backslash G) k_\nu(1) + \int_{\Gamma_n\backslash G} \sum_{\substack{\gamma\in \Gamma_n\\\gamma \neq 1}} k_\nu(g^{-1}\gamma g)\, dg,
\]
where $\Lambda_n$ denotes the spectrum of $\Gamma_n\backslash G$ (each $\lambda$ appearing with its respective multiplicity). 

To deal with the second term, we use part \eqref{size-k} of Lemma \ref{lemma:test-fn}, as well as the support condition on $k_\nu$, to deduce
\[
\int_{\Gamma_n\backslash G} \sum_{\substack{\gamma\in \Gamma_n\\\gamma \neq 1}} k_\nu(g^{-1}\gamma g)\, dg\ll \tilde\beta(\nu)\int_{\Gamma_n\backslash G} |\{\gamma\in \Gamma_n\setminus\{1\}: d(g,\gamma g)\leq 1\}|dg.
\]
As $\Gamma_n$ is torsion free, the inner sum is empty for all $g\in (\Gamma_n\bs G)_{>1}$, so that it suffices to bound
\[
\int_{(\Gamma_n\backslash G)_{\leq 1}} N_R(g)dg,
\]
where
\begin{equation}\label{defn-NRg}
N_R(g)=|\{\gamma\in \Gamma_n: d(g,\gamma g)\leq 1\}|.
\end{equation}
For $g\in (\Gamma_n\bs G)_{\leq 1}$ we apply \cite{AB+}*{Lemma 6.18}, which provides a constant $C>0$, depending only on $G$, such that $N_R(g)\leq C \InjR_{\Gamma_n}(g)^{-\dim S}$. The uniform discreteness of the (torsion free)  $\Gamma_n$ implies that $\InjR_{\Gamma_n}(g)^{-1} \ll 1$, uniformly in $n$ and $g$. Taking these estimates together we get
\begin{equation}\label{rough-smooth-count}
\sum_{\lambda\in \Lambda_n}h_\nu(\lambda) 
=\vol(\Gamma_n\backslash G)k_\nu(1) +O(\tilde\beta(\nu)\vol(\Gamma_n\backslash G)_{\leq 1}).
\end{equation}

Furthermore, we have
\begin{equation}\label{identity-size}
k_\nu(1)\asymp_\epsilon \tilde\beta(\nu),
\end{equation}
where only the lower bound depends on $\varepsilon$. Indeed, the upper bound results from part \eqref{size-k} of Lemma \ref{lemma:test-fn}. To obtain the lower bound, one applies the Plancherel inversion formula \eqref{Plancherel-inversion}, and parts \eqref{1} and \eqref{size-h} of Lemma \ref{lemma:test-fn}, to get 
\[
k_\nu(1)\geq \int_{\substack{\lambda\in i\af^*\\ \|\lambda-\nu\|_2\leq\delta}} h_\nu(\lambda)\beta(\lambda)d\lambda\geq  c_2\int_{\substack{\lambda\in i\af^*\\ \|\lambda-\nu\|_2\leq\delta}}\beta(\lambda)d\lambda,
\]
for any $0<\delta<c_1$. Recall that $\nu\in i\af^*$ is $\varepsilon$-regular. Taking $\delta$ small enough, we may assume that the $\lambda\in i\af^*$ such that $\|\lambda-\nu\|_2\leq\delta$ are $\varepsilon/2$-regular. Then the lower bound follows from the inequality $\beta(\lambda)\gg_\varepsilon \tilde\beta(\lambda)$ for such $\lambda$; see \cite[(3.44a)]{DKV}.

From \eqref{rough-smooth-count} and \eqref{identity-size} it follows that
\[
\sum_{\lambda\in \Lambda_n}h_\nu(\lambda) \asymp_\varepsilon \vol(\Gamma_n\backslash G)\tilde\beta(\nu)\left(1+O\left(\frac{\vol(\Gamma_n\backslash G)_{\leq 1}}{\vol(\Gamma_n\backslash G)}\right)\right),
\]
where only the lower bound depends on $\varepsilon$. By the Benjamini--Schramm assumption, we have $\vol(\Gamma_n\backslash G)_{\leq 1}=o(\vol(\Gamma_n\backslash G))$ as $n\rightarrow\infty$. Thus, for $n$ large enough we have
\begin{equation}\label{smooth-count}
\sum_{\lambda\in \Lambda_n}h_\nu(\lambda) \asymp_\varepsilon \vol(\Gamma_n\backslash G)\tilde\beta(\nu)
\end{equation}
For the first bound in \eqref{more-precise}, it suffices at this point to take $\varrho_1=c_1$, apply parts \eqref{1} and \eqref{size-h} of Lemma \ref{lemma:test-fn}, and then quote the upper bound in \eqref{smooth-count} to get
\[
N(B(\nu,c_1),\Gamma_n)\leq c_2 \sum_{\substack{\lambda\in\Lambda_n\\ \|\lambda-\nu\|_2\leq c_1}}h_\nu(\lambda)\leq c_2\sum_{\lambda\in \Lambda_n}h_\nu(\lambda) \ll \vol(\Gamma_n\backslash G)\tilde\beta(\nu).
\]
For the second bound in \eqref{more-precise}, we must show that the left-hand side in \eqref{smooth-count} approximates $N(B(\nu,\varrho),\Gamma_n)$, for some $\varrho>0$.

A crucial ingredient for passing from a smooth count as above to a sharp count will be a good upper bound on $N(B(\mu,t),\Gamma_n)$, for any center $\mu\in i\af^*$ and any $t\geq 1$. This is proved similarly to the preceding analysis. Indeed, from Lemma \ref{lemma:test-fn}, together with the preceding geometric argument, we obtain
\[
\sum_{\substack{\lambda\in\Lambda_n\\\|{\rm Im}\,\lambda-\mu\|_2\leq t}} 1\ll  
 \sum_{\lambda\in\Lambda_n}h_{\mu,c_1^{-1}t}(\lambda)=\vol(\Gamma_n\backslash G) k_{\mu,c_1^{-1}t}(1) + o\big(\vol(\Gamma_n\backslash G)\tilde\beta(\mu) \big).
\]
Now $k_{\mu,t}(1)\ll \tilde\beta(\mu)t^r$ from part \eqref{size-k} of Lemma \ref{lemma:test-fn}. We conclude that (for $n$ large enough)
\begin{equation}\label{Weyl-upper}
N(B'(\mu,t),\Gamma_n)\ll \vol(\Gamma_n\backslash G)\tilde\beta(\mu)t^r,
\end{equation}
 where $B'(\mu,t)=\{\lambda\in \af^*_\C\mid \|{\rm Im}\, \lambda - \mu\|_2\le t\}$.
We now return to the left-hand side of \eqref{smooth-count}. We first claim:

\begin{lemma}\label{lem:away:from:center}
For every $\varrho>1$,
\begin{equation}\label{away-from-center}
\sum_{\substack{\lambda\in\Lambda_n\\ \|{\rm Im}\, \lambda-\nu\|_2>\varrho}}h_\nu(\lambda) \ll_A \vol(\Gamma_n\backslash G)\tilde\beta(\nu)\varrho^{-A}.
\end{equation}
\end{lemma}
\begin{proof}

To prove \eqref{away-from-center} we cover $\{\lambda\in i\af^*: \|\lambda-\nu\|_2>\varrho\}$ by balls of unit radius with centers based at an affine lattice $\Lambda\subset i\af^*$ containing $\nu$. Let $A(\nu,\varrho)$ denote the set of all $\lambda\in \Lambda$ such that the ball of radius $1$ in $i\af^*$ around $\lambda$ is entirely contained in the ball of radius $\varrho$ in $i\af^*$ around $\nu$.

For any fixed $\mu\in\Lambda-A(\nu,\varrho)$, we bound the contribution of $\lambda\in\Lambda_n$ for which ${\rm Im}\,\lambda\in B(\mu)$ by using part \eqref{rapid-decay} of Lemma \ref{lemma:test-fn} to get
\[ 
\sum_{\substack{\lambda\in\Lambda_n\\ \|{\rm Im}\,\lambda-\mu\|_2\leq 1}}h_\nu(\lambda) \ll_A (1+\|\mu-\nu\|_2)^{-A}N(B'(\mu,1),\Gamma_n).
\]
Using \eqref{Weyl-upper} with $t=1$ and summing over $\mu\in\Lambda- A(\nu, \varrho)$ we get
\[
\sum_{\substack{\lambda\in\Lambda_n\\ \|{\rm Im}\,\lambda-\nu\|_2>\varrho}}h_\nu(\lambda)\ll_A \vol(\Gamma_n\backslash G)\sum_{\mu\in\Lambda- A(\nu, \varrho)} (1+\|\mu-\nu\|_2)^{-A}\tilde\beta(\mu).
\]
Now by \eqref{beta-tilde-bound} we have $\tilde\beta(\mu)\ll (1+\|\mu-\nu\|_2)^{\dim\nf}\tilde\beta(\nu)$, so that
\[
\sum_{\|{\rm Im}\, \lambda-\nu\|_2>\varrho}h_\nu(\lambda)\ll_A \vol(\Gamma_n\backslash G)\tilde\beta(\nu)\sum_{\mu\in\Lambda-A(\nu,\varrho)}(1+\|\mu-\nu\|_2)^{-A}.
\]
For $\varrho$ sufficiently large, this last sum is at most $O_A(\varrho^{-A})$, as desired.
\end{proof}

From this lemma, we want to deduce the following estimate:
\begin{equation}\label{away-from-center1}
\sum_{\substack{\lambda\in\Lambda_n\\ \|\lambda-\nu\|_2>\varrho}}h_\nu(\lambda) \ll_A \vol(\Gamma_n\backslash G)\tilde\beta(\nu)\varrho^{-A}.
\end{equation}
For this we split the sum on the left hand side into three parts,
\[
\sum_{\substack{\lambda\in\Lambda_n\\ \|{\rm Im}\, \lambda-\nu\|_2>\varrho}}h_\nu(\lambda)
+ \sum_{\substack{\lambda\in\Lambda_n\\  \varrho^2-\|\rho\|_2^2< \|{\rm Im}\,\lambda-\nu\|_2^2\le \varrho^2, \\ \|{\rm Re}\, \lambda\|_2^2> \varrho^2-\|{\rm Im}\,\lambda-\nu\|_2^2  }}h_\nu(\lambda)
+ \sum_{\substack{\lambda\in\Lambda_n\\ \|{\rm Im} \lambda-\nu\|_2^2\le \varrho^2-\|\rho\|_2^2,\\ \|{\rm Re}\, \lambda\|_2^2> \varrho^2-\|{\rm Im}\,\lambda-\nu\|_2^2  }}h_\nu(\lambda).
\]
Using Lemma \ref{lem:away:from:center} the first sum can be bounded by $\ll_A \vol(\Gamma_n\backslash G) \tilde\beta(\nu) \varrho^{-A}$. The third sum is in fact empty: the conditions on the real and complex part of $\lambda$ imply that $\|\Re\lambda\|_2> \|\rho\|_2$, which is not possible for $\lambda$ in the spectrum of $L^2(\Gamma_n\backslash G)$ by \eqref{eq:unitary:rep}. Finally, for the second term, we extend the sum to all $\lambda\in \Lambda_n$ with $\|{\Im }\, \lambda-\nu\|_2^2\ge \varrho^2-\|\rho\|_2^2$. This is possible as $h_\nu(\lambda)$ is non-zero on the unitary spectrum by Lemma \ref{lemma:test-fn}. Hence we can use Lemma \ref{lem:away:from:center} to bound this term by $\ll_A (\varrho^2-\|\rho\|_2^2)^{-A/2}$.

We deduce from the lower bound in \eqref{smooth-count}, as well as \eqref{away-from-center1} that, with $A>1$ fixed and $\varrho$ large enough with respect to the implied constant,
\[
\sum_{\substack{\lambda\in \Lambda_n\\\|\lambda-\nu\|_2\leq \varrho}}h_\nu(\lambda)\gg_\varepsilon \vol(\Gamma_n\backslash G) \tilde\beta(\nu).
\]
On the other hand, by applying part \eqref{size-h} of Lemma \ref{lemma:test-fn} we have
\[
N(B(\nu,\varrho),\Gamma_n)\gg \sum_{\substack{\lambda\in\Lambda_n\\ \|\lambda-\nu\|_2\leq \varrho}}h_\nu(\lambda). 
\]
Combing these yields the second bound in \eqref{more-precise}, with $\varrho_2=\varrho$. \qed

\section{Spectral side}\label{sec:spectral}

We now return to the setting of $G=\SL_d(\R)$ and prove Theorem \ref{thm:spec}. In the course of the proof, we will make use of both Appendices \ref{appendix:X0} and \ref{sec:cones}.

We first examine how $k_t$ acts on eigenfunctions $\psi_\lambda$. The Harish-Chandra transform of $k_t$ is given by
\begin{equation}\label{sph-inv}
h_t(\lambda)=\frac{1}{\sqrt{m_G(E_t)}}\int_{E_t} \varphi_\lambda(g)\, dg.
\end{equation}
Then, recalling the definition \eqref{defn-propogator}, we have
\[
U_t\psi_\lambda=h_t(\lambda)\psi_\lambda,
\]
which follows from the doubling formula of elementary spherical functions. 

Let $\mathbf{A}(\tau)$ be the time averaging operator of \eqref{defn-Abar}. Since $k_t$ is self-adjoint we have
\[
\bigg|\langle \mathbf{A}(\tau)\psi_\lambda , \psi_\lambda\rangle_{L^2(Y)}\bigg|^2=\bigg(\frac{1}{\tau}\int_0^\tau |h_t(\lambda)|^2\, dt\bigg)^{\!\! 2}\,\bigg|\langle a\psi_\lambda, \psi_\lambda\rangle_{L^2(Y)}\bigg|^2.
\]
To prove Theorem \ref{thm:spec}, it will therefore be enough to show that $|h_t(\lambda)|$ is bounded away from $0$ on average over $t$. More precisely, we prove the following higher rank generalization of \cite[\S 8.1]{ABL}. 

\begin{proposition}\label{prop:average:lower:bound}
Let $G=\SL_d(\R)$. Given a compact set $\Omega\subseteq i\af^*$, there exist constants $C,\tau_0>0$, depending on $\Omega$, such that, for every $\tau\ge \tau_0$ and $\lambda\in\Omega$, 
 \[
\frac{1}{\tau}\int_0^\tau |h_t(\lambda)|^2\, dt \ge C.
 \]
\end{proposition}

\subsection{Main idea of proof}\label{subsection:ideas}

We begin by writing the integral defining $h_t$ in polar coordinates, according to the Cartan decomposition. Recall the averaging set $E_t$ from \S\ref{sec:Et} and its polar decomposition \eqref{Et-polar}. Letting $f_\lambda(X)=e^{\langle \rho,X\rangle}\varphi_\lambda(e^X)$, and using the Cartan measure decomposition in \eqref{Cartan-measure-decomp}, we find that \eqref{sph-inv} becomes
\begin{equation}\label{sph-inv2}
h_t(\lambda)=\frac{c_C}{\sqrt{m_G(E_t)}}\int_{\mathscr{P}_t^+} f_\lambda(e^X)J(X)e^{-\langle\rho,X\rangle}dX.
\end{equation}

The idea of the proof of Proposition \ref{prop:average:lower:bound} is to replace $f_\lambda$ by the main term $\Phi_\lambda$, discussed below, of its Harish-Chandra asymptotic expansion relative to the Levi determined by $X^0$, defined in \eqref{defnX0}. We do this by showing, in Appendix \ref{appendix:X0} (see Lemma \ref{mu-basis}), that $\mathscr{P}^+$ is 
 the intersection with $\overline{\af^+}$ of a translated cone $\mathscr{C}=X^0+\CmC^0$ in $\af$. We then show that,  when $\lambda$ is taken to be rational, $\Phi_\lambda$ behaves like a sum of characters in the direction of $X^0$. An argument using linear independence of characters then yields the result. We carry out this argument in this section, using some combinatorial results which we establish in Appendix \ref{sec:cones}. 

\subsection{The main term and periodicity}\label{sec:MT-and-periodicity}

We wish to replace $f_\lambda$ by the main term in its Harish-Chandra expansion. This is possible thanks to Proposition \ref{prop:GV2} below. To state it, we must introduce some notation, beyond that of \S\ref{sec:gen-notation}. For simplicity, we continue to take $G=\SL_d(\R)$, although Proposition \ref{prop:GV2} applies more generally.

A Levi subgroup of $G$ will be called \textit{semistandard} if it contains $A$. Let $\Delta_0\subseteq \Delta$ be a (possibly empty) subset of the set of simple roots, and let $\Phi^+_0$ be the subset of $\Phi^+$ generated by $\Delta_0$. The sets $\Delta_0$ correspond to semistandard Levi subgroups $M\subseteq G$ such that $\Phi^+_0$ is the set of positive roots of $A$ on $M$. With this identification $\Delta_0=\emptyset$ corresponds to $M=A$ and $\Delta_0=\Delta$ to $M=G$.

For fixed $\Delta_0$ and corresponding $M$ we further introduce the following notation:
\begin{itemize}
\item  $\rho^M$ is the half-sum of all $\alpha\in\Phi_0^+$, 
\item $W^M$ is the Weyl group of $A$ in $M$, 
\item $\mathbf{c}^M$ denotes the Harish-Chandra $\mathbf{c}$-function for $\Phi_0^+$, 
\item $\varphi^M_\lambda$ is the spherical function on $M$ for $\lambda\in\af_\C^*$, that is,
\[
 \varphi^M_\lambda(m)= \int_{K\cap M} e^{\langle \lambda + \rho^M, H_0(km)\rangle}\, dm.
\]
\item for $X\in \af$, $\lambda\in\af_\C^*$, and $t>0$ let $f_\lambda^M(t,X)= e^{\langle \rho^M,tX\rangle} \varphi_\lambda^M(e^{tX})$ and (setting $t=1$) let
\[
f_\lambda^M(X)=e^{\langle \rho^M,X\rangle} \varphi_\lambda^M(e^X).
\]
\item $\af_M$ is the common null space of all $\alpha\in\Phi_0^+$, and $\af^M$ is its orthocomplement in $\af$ relative to the inner product $\langle X, Y\rangle=XY$.
\end{itemize}
We follow the standard practice of omitting the superscript $M$ from the notation whenever $M=G$. In this way we recover the $f_\lambda(X)$ introduced earlier. Other than $G$, the only other Levi which will be of importance to us throughout this section is the centralizer of $X^0$ in $G$. \textit{Henceforth shall take $M$ to be the centralizer of $X^0$ in $G$, and $\Delta_0$ such that it corresponds to $M$.} In this case $X^0\in\af_M$.

For $\lambda\in i\af^*$ regular we define
\begin{equation}\label{Phi-lambda-def}
\Phi_\lambda(X) = \sum_{w\in W^M\backslash W} \frac{\mathbf{c}(w\lambda)}{\mathbf{c}^M(w\lambda)} f^M_{w\lambda}(X).
\end{equation}
In fact, we can define $\Phi_\lambda(X)$ for any regular $\lambda\in\af^*_\C$ for which the $\mathbf{c}(w\lambda)$ do not have a pole. We can make this more precise, as follows. For a real number $\eta>0$ let $A_\eta\subseteq \af^*$ denote the convex hull of the Weyl group orbit of $\eta\rho$. Define $\cT_\eta\subseteq \af^*_\C$ to be the set of all $\lambda\in \af^*_\C$ with $\Re\lambda\in A_\eta$. Let $\eta$ be so small that for every positive root $\alpha$ and all $\lambda\in\cT_\eta$ we have $\langle\Re\lambda, \alpha^\vee\rangle\not\in \Z\smallsetminus\{0\}$. We fix such a $0<\eta<1$ from now on (we will make our choice more precise later). Note that $\cT_\eta$ is $W$-invariant by construction. We may extend the definition of $\Phi_\lambda(X)$ in \eqref{Phi-lambda-def} to $\lambda\in \cT_\eta$ regular and $X\in\af$. Indeed, the $\mathbf{c}$-function is non-zero on  $\af^*_\C$ and does not have any poles for regular $\lambda\in \cT_\eta$ by our choice of $\eta$.

The relevance of $\Phi_\lambda(X)$ is apparent from the next result, due to Trombi and Varadarajan \cite[Theorem 2.11.2]{TV}, providing the full asymptotic expansion of the spherical function. We have taken the exact statement from \cite[{Theorem 5.9.4}]{GaVa}. This extends to all $X\in\af$ the asymptotic expansion proved by Harish-Chandra and Gangolli for conical subsets of regular elements.

\begin{proposition}\label{prop:GV2}
Let $\Omega\subseteq i\af^*$ be a compact set and $c>0$. Then there exists $C>0$ and $m\ge 0$ such that for all $H\in\overline{\af^+}$ with $\beta_{\Delta_0}(H)\ge c$, where
\[
\beta_{\Delta_0}(H)=\min_{\alpha\in \Delta\smallsetminus \Delta_0} \alpha(H),
\]
and for every regular $\lambda\in \Omega$, we have
 \[
  \left|f_\lambda(H) - \Phi_\lambda(H) \right|
  \le C (1+\|H\|_2)^m e^{-2\beta_{\Delta_0}(H)}.
 \]
\end{proposition}

We now describe the key periodicity property of $\Phi_\lambda(X)$. For this we need the following definition.

\begin{definition}\label{def-rational}
We say that $\lambda\in i\af^*$ is \textbf{rational} if $i\langle w\lambda, X^0\rangle\in\Q$ for all $w\in W$. 
\end{definition}

If $\lambda$ is rational, we can choose $\tau_1\in \R_{>0}$ such that $\tau_1 \langle w\lambda, X^0\rangle\in 2i\pi\Z$ for all $w\in W$. Recall the definition of $\cC$ from Lemma \ref{mu-basis}.

\begin{lemma}\label{lem:periodic}
Let $\lambda\in i\af^*$ be rational and let $\tau_1$ be defined as above. Set $\tau_n = n\tau_1$ for $n\in \N$. Then for every $Y\in\cC$, $t\ge0$ and $n\in \N$ we have 
\[
\Phi_\lambda(Y + (t+\tau_n)X^0)
= \Phi_\lambda(Y+ tX^0).
\]
\end{lemma}
\begin{proof}
By definition of $\tau_1$ it follows that $e^{\langle \lambda, (t+\tau_n)X^0\rangle} =e^{\langle \lambda, tX^0\rangle}$ for every $n\in\N$. Further note that, since $\langle\rho^M,X^0\rangle=0$ and $H_0(ke^{Y+tX^0})=tX_0+H(ke^Y)$ for $k\in K\cap M$, we have
\[
\varphi_\lambda^M(e^{Y+tX^0})=e^{\langle \rho^M+\lambda,tX^0\rangle}\varphi_\lambda^M(e^Y)=e^{\langle \lambda,tX^0\rangle}\varphi_\lambda^M(e^Y).
\]
Thus
\begin{equation}\label{fM-periodicity}
f^M_\lambda(Y+ tX^0) = e^{\langle \lambda, tX^0\rangle} f^M_\lambda(Y).
\end{equation}
Applying these two formulas for every $w\lambda$, $w\in W^M\backslash W$, yields the assertion.
\end{proof}

\subsection{Convergence of an integral}
To take advantage of Lemma \ref{lem:periodic}, we shall need to replace the integration domain $\mathscr{P}_t^+$ in \eqref{sph-inv2} by the translated cone $\mathscr{C}$. Note that the Jacobian factor \eqref{first-J(X)} can be alternatively written as
\begin{equation}\label{J(H)}
J(X)=2^{-|\Phi^+|}\sum_{w\in W} \sigma(w) e^{2\langle w\rho, X\rangle} \qquad (X\in\overline{\af^+}),
\end{equation}
where $\sigma(w)\in\{\pm 1\}$ denotes the signature of $w$.
With this in mind, for $t>0$ and $\lambda\in \cT_\eta$, we put
\[
I(t,\lambda)=\int_{\cC^0} \Phi_\lambda(Y + tX^0) e^{\langle \rho, Y\rangle}\, dY,
\]
whenever it converges.

In this paragraph we address this question of convergence. From the defining expression for $\Phi_\lambda$ in \eqref{Phi-lambda-def}, it will be enough to understand the convergence of the integral
\[
J(t,\lambda)
=\int_{\cC^0}  f^M_{\lambda}(Y+ tX^0) e^{\langle \rho, Y\rangle}\, dY.
\]

\begin{lemma}\label{lem:wInt:nonzero}
There is $0<\eta<1$ such that, for every $\lambda\in \cT_\eta$ and $t\ge0$, the integral defining $J(t,\lambda)$ converges absolutely. 
\end{lemma}

\begin{proof}
We first observe that it suffices to consider $t=0$. Indeed, from \eqref{fM-periodicity} it follows that
\begin{equation}\label{eq:covariance:I}
J(t,\lambda) = e^{ t \langle \lambda, X^0\rangle} J(\lambda).
\end{equation}
where we have put $J(\lambda)=J(0,\lambda)$.

Let $Y\in \cC^0$ and write $Y= Y_M+Y^M$ with $Y_M\in \af_M$ and $Y^M\in \af^M$. Then
\begin{equation}\label{eq:fM-bound}
f^M_{\lambda}(Y)
= e^{\langle \rho^M, Y\rangle} \varphi_\lambda^M(e^Y)
= e^{\langle\lambda, Y_M\rangle} e^{\langle \rho^M, Y^M\rangle} \varphi_\lambda^M(e^{Y^M}).
\end{equation}
Proposition \ref{claims} then immediately yields the bound $e^{\langle\lambda, Y_M\rangle}\leq e^{-\eta C\langle\rho,Y\rangle}$ on the first factor on the right-hand side of \eqref{eq:fM-bound}. 

To bound the remaining two factors, we first let $Y^M_+$ be a point in the $W^M$-orbit of $Y^M$ lying in $\overline{\af^{M,+}}$. Then by Proposition 2.2 and Theorem 2.8 of \cite[Ch. X, \S 2]{JoLa} we have 
\begin{align*}
\big|\varphi_\lambda^M(e^{Y^M})\big| 
\le \varphi_{\Re \lambda}^M(e^{Y^M})
&= \varphi_{\Re \lambda}^M(e^{Y^M_+})\\
&  \le \max_{w\in W^M} e^{\langle w{\rm Re}\, \lambda, Y^M_+\rangle} \varphi_0^M(e^{Y^M_+}) \\
& \ll  \max_{w\in W^M} e^{\langle w{\rm Re}\, \lambda, Y^M_+\rangle} e^{-\langle \rho^M, Y^M_+\rangle} (1+\|Y^M_+\|_2)^{|\Phi^+|}.
\end{align*}
Now $\|Y^M_+\|_2 = \|Y^M\|_2$ and $\langle\rho^M, Y^M\rangle\le \langle \rho^M, Y^M_+\rangle$. Thus there is a $c>0$ such that
\[
 \big|e^{\langle \rho^M, Y^M\rangle} \varphi_\lambda^M(e^{Y^M})\big| 
 \leq c \max_{w\in W^M} e^{\langle w\Re \lambda, Y^M_+\rangle} (1+\|Y^M\|_2)^{|\Phi^+|}.
\]
Using the $W$-invariance of $\langle\cdot,\cdot\rangle$ and $Y^M= Y-Y_M$, we obtain
\[
 \max_{w\in W^M} \langle w\Re \lambda, Y^M_+\rangle=\max_{w\in W^M} \langle w\Re \lambda, Y^M\rangle \leq \max_{w\in W} |\langle w \Re \lambda, Y\rangle|+\max_{w\in W} |\langle w \Re \lambda, Y_M\rangle|.
\]
Note that $w\Re\lambda=\Re w\lambda$ and $w\lambda\in\mathcal{T}_\eta$. Using Remark \ref{rem:positive:rho}, we have $ |\langle \Re \lambda, Y\rangle|\ll -\eta\langle \rho, Y\rangle$ for all $\lambda\in \cT_\eta$ and $Y\in\mathscr{C}^0$. From this, and Proposition \ref{claims}, we deduce that there is a constant $C'>0$, depending only on $d$, such that 
\[
\max_{w\in W^M} \langle w\Re \lambda, Y^M_+\rangle 
 \le - \eta C'\langle\rho, Y\rangle.
\]
Similarly, we can find $c>0$ such that $\|Y^M\|_2 \le c \|Y\|_2$. Hence if $\eta>0$ is sufficiently small, then it follows from \eqref{eq:fM-bound} and our estimates above that for some $c'>0$,
\[
\left|f^M_{\lambda}(Y)\right|
\le c' e^{-\frac{1}{2}\langle \rho, Y\rangle}(1+\|Y\|_2)^{|\Phi^+|}.
\]
We therefore obtain,
\[
\int_{\cC^0}  \left|f^M_{\lambda}(Y) e^{\langle \rho, Y\rangle}\right|\, dY
\le \int_{\cC^0}  e^{\langle \frac{1}{2} \rho, Y\rangle}  (1+\|Y\|_2)^{|\Phi^+|}\, dY.
\]
Using Remark \ref{rem:positive:rho} we see that this last integral is finite. Thus the integral defining $J(\lambda)$, and hence $J(t,\lambda)$, converges absolutely.
\end{proof}

It follows that for $\lambda\in \cT_\eta$ we have
\begin{equation}\label{lin-comb-chars}
I(t,\lambda) = \sum_{w\in W^M\backslash W} \frac{\mathbf{c}(w\lambda)}{\mathbf{c}^M(w\lambda)} J(w\lambda) e^{t \langle w\lambda , X^0\rangle},
\end{equation}
where, we recall, $J(\lambda)=J(0,\lambda)$.
\subsection{Main term replacement}

Having established that $I(t,\lambda)$ is well-defined, we now show, using Proposition \ref{prop:GV2}, that $h_t(\lambda)$ can be approximated by $I(t,\lambda)$. 

\begin{proposition}\label{replacement}
There are $c, C>0$ such that, for regular $\lambda\in i\af^*$, 
\[
h_t(\lambda)=C I(t,\lambda)+O_\lambda(e^{-ct}).
\]
\end{proposition}
Before we prove this proposition, we establish the following lemma:

\begin{lemma}\label{lem:regular:polytope}
Let $0<\eta<1$ and define $\cE= X^0 -\sum_{i=1}^r [0,\eta]\beta_i^\vee = \sum_{i=1}^r [1-\eta, 1]\beta_i^\vee$ and $\cE^+ = \cE\cap \cP^+$. 
Then if $\eta$ is sufficiently small, we have:
\begin{enumerate}[label=(\roman{*})]
 \item For every $V$ in the closure $\overline{\cP^+\smallsetminus \cE^+}$ of  $\cP^+\smallsetminus \cE^+$ we have
 \begin{equation}\label{eq:bounded:from:X0}
  \langle \rho, V\rangle \le \langle\rho, X^0\rangle -\delta
 \end{equation}
for some $\delta=\delta(\eta)>0$ uniform in all $V$.

\item  Let $\Delta^M\subseteq \Delta$ denote the subset of simple roots in $M$. Then there exists $c=c(\eta)>0$ such that for every $X\in \cE$ and $\alpha\in \Delta\smallsetminus\Delta^M$ we have $\alpha(X)\ge c$.
\end{enumerate}
\end{lemma}

\begin{proof}
Since $X^0$ is the unique point in $\cP^+$ at which the linear map $X\mapsto \langle\rho, X\rangle$ attains its maximum, and $\overline{\cP^+\smallsetminus\cE^+}$ is closed, it will suffice to show that every point in $\overline{\cP^+\smallsetminus\cE^+}$ is bounded away from $X^0$. But this is clear, since $\cP^+ =(X^0 -\sum_{i=1}^r [0,\infty)\beta_i^\vee)\cap \af^+$ by Lemma \ref{mu-basis}. Note that this argument holds for any $0<\eta<1$ with the resulting constant $\delta$ of course depending on $\eta$.

For the second part, we note that for every $\alpha\in \Delta\smallsetminus \Delta^M$ we have $\alpha(X^0) = 2$. Hence if $\eta$ is chosen sufficiently small, each such $\alpha$ stays bounded away from $0$ on $\cE$ so that $(ii)$ follows.
\end{proof}

\begin{proof}[Proof of Proposition \ref{replacement}]
Let
\[
G(t,\lambda)=e^{-t\langle\rho, X^0\rangle} \int_{\cP^+_t} f_\lambda(X) e^{\langle \rho, X\rangle} \, dX.
\]
From \eqref{sph-inv2} and \eqref{J(H)}, as well as the volume computation in Corollary \ref{rem:volume}, we have
\[
h_t(\lambda)=C G(t,\lambda)+O(e^{-c_1t}),
\]
for some $c_1, C>0$. 
Next, let $\cE^+$ be defined as in Lemma \ref{lem:regular:polytope} with $\eta$ sufficiently small. Put
\[
H(t,\lambda)=  \int_{t\cE^+} f_\lambda(X) e^{\langle\rho, X- tX^0\rangle}\, dX. 
\]
Then $G(t, \lambda)=H(t,\lambda) + O(e^{-c_2t})$ for some absolute constant $c_2>0$ by Lemma \ref{lem:regular:polytope}$(i)$.

Note that for any $X\in t\cE$ we have $\alpha(X)\ge t/2$ for $\alpha\in\Delta\smallsetminus \Delta^M$, where $\Delta^M\subseteq \Delta$ denotes the subset of simple roots in $M$. 

Using Lemma \ref{lem:regular:polytope}$(ii)$ and Proposition \ref{prop:GV2}, there exists $c_3>0$ such that 
\[
f_\lambda(X) - \Phi_\lambda(X) \ll_\lambda e^{-c_3t}
\]
for every $X\in t\cE^+$ and every regular $\lambda\in i\af^*$. Hence, for regular $\lambda\in i\af^*$, we have
\begin{align*}
H(t,\lambda)&=  \int_{t\cE^+} \Phi_\lambda(X) e^{\langle\rho, X- tX^0\rangle}\, dX + O_\lambda(e^{-c_3t})\\
&= \int_{t\cE^+- tX^0} \Phi_\lambda(Y + tX^0) e^{\langle \rho, Y\rangle}\, dY+O_\lambda(e^{-c_3t}),
\end{align*}
after a change of variables. Noting that $t\cE^+- tX^0\subseteq \cC^0$, we can then extend the integral to all of $\cC^0$, incurring an additional error of $O_\lambda(e^{-c_4t})$, for some $c_4>0$.
\end{proof}

\subsection{A non-vanishing result}
We now show that $I(t,\lambda)$ is generically non-vanishing on $i\af^*$. This will follow from the absolute convergence of $I(t,\lambda)$, the periodic behavior of $\Phi_\lambda$ along $X^0$ for rational $\lambda$, and a linear independence of characters argument.

\begin{proposition}\label{prop:nonzero}
There exists a finite number of hyperplanes $\cH_1,\ldots, \cH_N\subseteq i\af^*$ and a discrete set $\cD\subset i\af^*$ such that for every
\[
\lambda\in i\af^*\smallsetminus (\cH_1\cup\cdots\cup\cH_N\cup\cD),
\]
there exists $t=t_\lambda \ge0$ with $I(t,\lambda)\neq 0$.
\end{proposition}

\begin{proof}
We first claim that there exists a discrete subset $A\subseteq \cT_\eta$ such that $J(t,w\lambda)\neq 0$ for all $\lambda\in \cT_\eta\smallsetminus A$, $w\in W$, and $t\ge0$. Indeed, for fixed $m\in M$, the map $\lambda\mapsto \phi_\lambda^M(m)$ is holomorphic in $\lambda\in \af^*_\C$. 
Let $\eta>0$ be sufficiently small and at most as large as in Lemma \ref{lem:wInt:nonzero}. Since the integral defining $J(t,\lambda)$ converges absolutely, $J(t,\lambda)$ is in fact a holomorphic function for $\lambda\in \cT_\eta$. 
From the non-vanishing at $\lambda=0$
\[
J(t,0)=\int_{\cC^0}  e^{\langle \rho+\rho^M, Y\rangle}\, dY \neq0,
\]
it follows that $J(t,\lambda)$ does not vanish identically in $\lambda$, hence there exists a discrete subset $A\subseteq \cT_\eta$ such that $J(t,\lambda)\neq 0$ for all $\lambda\in \cT_\eta\smallsetminus A$.

By our choice of $\eta$, $\mathbf{c}(w\lambda)/\mathbf{c}^M(w\lambda)$ has no pole for regular $\lambda$ in $\cT_\eta$. Let $\cT_\eta'$ denote the set of all regular $\lambda\in\cT_\eta\smallsetminus A$ such that $\langle \lambda, w^{-1}X^0\rangle\neq \langle \lambda, v^{-1}X^0\rangle$ for all $w, v\in W^M\backslash W$, $w\neq v$. Then $\cT_\eta'$ is in fact dense in $\cT_\eta$ since $w^{-1}X^0\neq v^{-1} X^0$ for all $w\neq v$, $w, v\in W^M\backslash W$, so that we can obtain $\cT_\eta'$ by removing the hyperplanes $\{\lambda\mid  \langle \lambda, w^{-1}X^0\rangle = \langle \lambda, v^{-1}X^0\rangle\}$ for each $v\neq w$ in $W^M\backslash W$.

Recall the expansion \eqref{lin-comb-chars}. By definition of $\cT_\eta'$, the phases $\langle w\lambda , X^0\rangle$, $w\in W^M\backslash W$, are pairwise distinct for $\lambda\in\cT_\eta'$. We can therefore apply \cite[Lemma 56]{HC58} (see also \cite[Ch.VIII, Lemma 0.1]{JoLa}) to the sum over $w\in W^M\backslash W$ to obtain for each $\lambda\in\cT_\eta'$, 
\[
\limsup_{t\rightarrow \infty} \left|I(t,\lambda)\right|^2
\ge \sum_{w\in W^M\backslash W} \left|\frac{\mathbf{c}(w\lambda)}{\mathbf{c}^M(w\lambda)} J(w\lambda)\right|^2.
\]
Now for each $w\in W$ and $\lambda\in \cT_\eta'$, the terms $\frac{\mathbf{c}(w\lambda)}{\mathbf{c}^M(w\lambda)} J(w\lambda)$ are non-zero, hence for each $\lambda\in \cT_\eta'$ there exists $t=t_\lambda$ such that $I(t, \lambda)\neq0$.
\end{proof}

\subsection{Final arguments}

We now finish the proof of Proposition \ref{prop:average:lower:bound}.

By enlarging $\Omega$ if necessary, we can assume that $\overline{\Omega^\circ} = \Omega$. Let $\cH_1,\ldots, \cH_N$ and $\cD$ be as in Proposition \ref{prop:nonzero}. If necessary, we add in more hyperplanes such that the complement of the union of them in $i\af^*$ consists of regular points only. Then the set of rational $\lambda$ in
\[
\Omega\cap \left(i\af^*\smallsetminus (\cH_1\cup\ldots\cup\cH_N\cup\cD)\right)
\]
is dense in $\Omega$. Since $\lambda\mapsto \frac{1}{T}\int_0^T |h_t(\lambda)|^2$ is a continuous function, and $\Omega$ is compact, Proposition \ref{prop:average:lower:bound} follows from the next result.

\begin{lemma}\label{cor:large:average:rational}
For every rational
\[
\lambda\in i\af^*\smallsetminus (\cH_1\cup\cdots\cup\cH_N\cup\cD)
\]
there exists $T_0>0$ depending on $\lambda$ such that for all $T>T_0$ we have
\[
\frac{1}{T}\int_0^T |h_t(\lambda)|^2\, dt >0.
\]
\end{lemma}

\begin{proof}
Recall the definition of rational from Definition \ref{def-rational}. Let $\tau_1\in\R_{>0}$ (depending on $\lambda$) and $\tau_n=n\tau_1$ for $n\in\N$ be as in the statement of Lemma \ref{lem:periodic}. By Proposition~\ref{prop:nonzero} we can choose $t_0\ge0$ such that $I(t_0,\lambda)=:a\neq0$. Write $t_n=t_0+\tau_n$ for $n\in\N$. Note that, by Lemma \ref{lem:periodic}, we have $I(t+\tau_n,\lambda)=I(t,\lambda)$ for all $t\ge0$ and all $n\in\N$. Thus $I(t_n,\lambda)=a$, and we can find a small $\eps>0$ such that for every $n\in\N$ and every $s\in (t_n-\eps, t_n+\eps)$ we have 
\[
|I(s,\lambda)|\ge |a|/2>0.
\]
From Lemma \ref{replacement} we may choose $T_0>0$ so large such that $|h_t(\lambda) - I(t,\lambda)|<|a|/4$ for all $t\ge T_0$. Let $N\in\N$ be such that $t_n-\epsilon\ge T_0$ for all $n\ge N$. Then for every $n\ge N$ we have 
\[
|h_s(\lambda)|\ge |a|/4
\]
for all $s\in (t_n-\eps, t_n+\eps)$. We therefore get for $T>2T_0 + \tau_1$ that
\[
\frac{1}{T}\int_0^T |h_t(\lambda)|^2\, dt
\gg  \frac{1}{T}\sum_{n=N}^{(T-1-t_0)/\tau_1} \int_{-\eps}^\eps |h_{t_n+s}(\lambda)|^2\, ds
\gg_{\lambda} \frac{T-T_0}{T} |a|^2>0,
\]
as we wanted to show.
\end{proof}

\section{Geometric side}\label{sec:geom}

The aim of this section is to prove Theorem \ref{thm:GeomSide}. The essential feature of the geometric side is the ergodic properties of the sets $gE_t\cap E_t$ in the quotient $\Gamma\backslash G$, as $t>0$ and $g\in G$ vary.

\subsection{A general bound}
In this section, we return to the general setting of Section \ref{sec:Weyl}, with $S=G/K$ and $\Gamma$ a given torsion free cocompact lattice in $G$. 

For $g,h\in G$ let $N_R(g,h)$ denote the number of
$\gamma\in \Gamma$ for which $d(g, \gamma h)\le R$. This generalizes the notation $N_R(g)$ from \eqref{defn-NRg} when $g=h$. Let $N_\Gamma(R)=\sup_{(g,h)}N_R(g,h)$.

\begin{lemma}\label{NGammaR}
There exists $c>0$, independent of $R$ and $\Gamma$, such that 
\[
N_\Gamma(R)\ll {\rm InjRad}(\Gamma\backslash S)^{-\dim S} e^{c R}
\]
\end{lemma}

\begin{proof}
Note that $N_R(g,h) =
N_R(g,\gamma h)$ for all $\gamma\in\Gamma$ so that we can assume that
$h$ is such that $d(g,h)=\min_{\gamma\in \Gamma} d(g, \gamma h)$. If
$d(g,h)>R$, then $N_R(g,h)=0$ so that we can assume $d(g, h)\le R$. 

Now if $\gamma\in\Gamma$ is such that $d(g,\gamma h)\le R$, then
\[
d(h,\gamma h) 
\le d(g,h)+ d(g,\gamma h)
\le 2R
\]
so that $N_R(g,h)\le N_{2R}(h,h)$. Applying \cite[Lemma 6.18]{AB+} to
$N_{2R}(h,h)$ gives the assertion of the lemma.
\end{proof}

The following lemma is an adaptation of \cite[Lemma 5.1]{LS} to our setting. Recall the geodesic balls $B_t$, first introduced in \eqref{Bt-polar} in the setting of $G=\SL_d(\R)$, but more generally defined for all $G$ using the same notation.
\begin{lemma}\label{lem:thick-thin}
There is $c>0$, depending only on $G$, such that the following holds. Let $R>0$. Let $K\in C(G\times G)$ be invariant under the diagonal action of $\Gamma$ and satisfy $\{g^{-1}h: (g,h)\in {\rm supp}(K)\} \subset B_R$. Let $\bf{K}$ be an integral operator on $L^2(\Gamma\bs G)$ with kernel $K$. Then
\[
\|\mathbf{K}\|_{\rm HS}^2\leq \int_{\Gamma\backslash G} \int_{G}|K(g,h)|^2 dg dh+O\left(\frac{e^{cR}}{{\rm InjRad}(\Gamma\bs S)^{\dim S}}\vol\left((\Gamma\backslash G)_{\leq 2R}\right)\|K\|_\infty^2\right).
\]
\end{lemma}

\begin{proof}
By definition, we have
\[
\big\| \mathbf{K}\big\|_{\textrm{HS}}^2=\int_{\Gamma\backslash G}\int_{\Gamma\backslash G} |\sum_{\gamma\in \Gamma}K(g,\gamma h)|^2 dg dh.
\]
Let $D\subset G$ be a fundamental domain for $\Gamma$. We decompose $D$ into its $2R$-thin and -thick parts $D=D_{\leq 2R}\cup D_{>2R}$, where, similarly to \eqref{R-thin-part}, we have put
\[
D_{\leq 2R}=\{x\in D: {\rm InjRad}_\Gamma (x)\leq 2R\}.
\]
From the support condition on $K$, it follows that if $K(g,\gamma h)\neq 0$ then $d(g,\gamma h)\leq R$. 

On the other hand, if $h\in D_{>2R}$, and $\gamma,\delta\in\Gamma$, $\gamma\neq\delta$, then $d(\gamma h, \delta h)>2R$. Hence if $d(g,\gamma h)\le R$ and $d(g, \delta h)\le R$ we get 
\[
 2R <d(\gamma h, \delta h)\le d(g, \gamma h) + d(g, \delta h)\le 2R
\]
which is a contradiction, proving that in the contribution of the $2R$-thick part. Further note that for fixed $g$, such $\gamma=\gamma_g\in \Gamma$ with $d(g,\gamma_g h)\le R$ is independent of $h$ as otherwise this would similarly lead to a contradiction since $D$ is connected. Hence using the left-invariance of the Haar measure on $G$, and the diagonal $\Gamma$-invariance of $K$, we can compute
\begin{align*}
\int_{D_{> 2R}}\int_D |\sum_{\delta\in\Gamma} K(g,\delta h)|^2 dg dh&=\int_{D_{> 2R}}\int_D |K(g,\gamma_g h)|^2 dgdh\\
&\leq \int_G\int_D |K(g,\gamma_g h)|^2 dg dh\\
&= \int_D \int_G  |K(g, g h)|^2 dh dg\\
&=\int_{\Gamma\backslash G}\int_G |K(g, gh)|^2dhdg\\
& =\int_{\Gamma\backslash G}\int_G |K(g, h)|^2dhdg.
\end{align*}
To deal with the $2R$-thin part, we use Cauchy--Schwarz and the support condition on $K$ to obtain
\[
\bigg|\sum_{\gamma\in\Gamma} K(g,\gamma h)\bigg|^2 \leq N_\Gamma(R)\sum_{\gamma\in\Gamma} |K(g,\gamma h)|^2.
\]
The factor $N_\Gamma(R)$ is bounded by Lemma \ref{NGammaR}. Furthermore,
\[
\int_{(\Gamma\backslash G)_{\leq 2R}}\int_{\Gamma\backslash G} \sum_{\gamma\in\Gamma} |K(g,\gamma h)|^2 dg dh\leq \sup_{g\in \Gamma \backslash G, h\in G}|K(g, h)|^2\int_{(\Gamma\backslash G)_{\leq 2R}}\int_{\Gamma\backslash G} \sum_{\gamma\in\Gamma} \mathbf{1}_{B_R}(g^{-1}\gamma h) dg dh.
\]
By unfolding and changing variables, we have
\[
 \int_{\Gamma\backslash G} \sum_{\gamma\in\Gamma} \mathbf{1}_{B_R}(g^{-1}\gamma h) dg= \int_G \mathbf{1}_{B_R}(g^{-1} h) dg=m_G(B_R)=c_C\int_{\mathscr{B}_R^+}J(X)dX,
\]
the last equality coming from the Cartan measure decomposition \eqref{Cartan-measure-decomp}. The Jacobian factor satisfies $J(X)\ll e^{2\langle \rho,X\rangle}\ll e^{2R\|\rho\|_2}$ for $X\in\mathscr{B}^+$, using \eqref{J(H)}. From this one deduces\footnote{One could also quote the more precise result of Knieper, recalled in \eqref{Knieper-bound}, but an exponential bound of any quality is all that is required in our application.} that $m_G(B_R)\ll R^r e^{2R\|\rho\|_2}$, which is enough to complete the proof.  
\end{proof}

\subsection{Description of the kernel of $\mathbf{A}(\tau)$}
We return to the setting of $G=\SL_d(\R)$. We would like to apply Lemma \ref{lem:thick-thin} to our operator $\mathbf{A}(\tau)$ on $L^2(\Gamma\backslash G)$ from \eqref{defn-Abar}. For this we shall need a description of its kernel, which we shall denote by $A(\tau)$; it is a continuous function on $\Gamma\bs (G\times G)$. 

\begin{lemma}\label{lem:kernel:A_X}
The kernel of $\mathbf{A}(\tau)$ is
\begin{equation}
A(\tau)(g,h)=\frac{1}{\tau}\int_0^\tau \frac{1}{m_G(E_t)}\int_{gE_t\cap hE_t} a(x)  dxdt.\label{def-Fbar}
\end{equation}
\end{lemma}

\begin{proof}
By definition, the action of $U_taU_t$ on $f\in L^2(\Gamma\backslash G)$ is given by
\[
(U_taU_t)f(g)=\frac{1}{m_G(E_t)}\int_{E_t}a(gh_1)\int_{E_t} f(gh_1 h_2)dh_2 dh_1.
\]
We set $x=gh_1$ and $h=xh_2$, apply Fubini and use $E_t^{-1}=E_t$, to obtain
\begin{align*}
(U_taU_t)f(g)&=\frac{1}{m_G(E_t)}\int_{gE_t} a(x)\int_{xE_t} f(h)dh dx\\
&=\int_G\left( \frac{1}{m_G(E_t)}\int_{gE_t\cap hE_t} a(x)  dx\right)f(h)dh.
\end{align*}
Averaging over $t$ yields the claim.
\end{proof}

\subsection{Support and sup of $\mathbf{A}(\tau)$}

In order to bound the second term in Lemma \ref{lem:thick-thin}, we shall need to estimate the support and the supremum of the kernel function $A(\tau)$. This is accomplished in the next result, which is the higher rank generalization of \cite[Lemma 26]{ABL}. We recall the norm $|\cdot |$ on $G=\SL_d(\R)$ introduced in \S\ref{sec:Et} and defining the sets $E_t$.

\begin{proposition}\label{cor:A-support}
There exists a constant $b>0$ depending only on $d$ such that $gE_t\cap hE_t$ is empty unless $|h^{-1}g|\le 2t+b$. In particular, $A(\tau)(g, h)=0$ unless $|h^{-1}g|\le 2\tau+b$. Further,
\[
\sup_{g, h\in G}|A(\tau)(g, h)|^2\le \|a\|_\infty^2.
\]
\end{proposition}

\begin{proof} We begin by establishing the following 

\medskip

\noindent {\sc Claim:} For $x\in\SL_d(\R)$ we have $|x_{ij}|\leq |x|$. Moreover, there exists a constant $c>0$ depending only on $d$, such that for every $x=(x_{ij})\in\SL_d(\R)$ and every $i\in\{1,\ldots ,d\}$ there exists $j\in\{1,\ldots, d\}$ with $|x_{ij}| \ge c e^{-|x|}$.

\medskip

\noindent {\sc Proof of claim:} Write $x= k e^Z l$ with $k=(k_{ij}), l = (l_{ij}) \in K=\SO(d)$ and $Z\in\overline{\af^+}$. Let $\kappa_i=(k_{i1},\ldots, k_{id})$ and $\lambda_j= (l_{1j},\ldots, l_{dj})\in\R^d$, for $i,j=1,\ldots,d$. Let $\langle\cdot, \cdot\rangle$ denote the standard inner product on $\R^d$. Then $\kappa_1,\ldots,\kappa_d$ and $\lambda_1,\ldots, \lambda_d$ are both orthonormal bases of $\R^d$, and $x_{ij}=\langle\kappa_i,\lambda_j e^Z\rangle$. In particular, by Cauchy--Schwarz we have
\[
|x_{ij}|=|\langle\kappa_i,\lambda_j e^Z\rangle|\leq \|\lambda_j e^Z\|^{1/2}\leq e^{\|Z\|_\infty}=|x|,
\]
proving the first statement.

For the second statement, we begin by observing that for every $j=1,\ldots ,d$ we have
\[
\langle\kappa_i, \kappa_i e^Z\rangle=\sum_{j=1}^de^{Z_j}k_{ij}^2\geq (\min_j e^{Z_j})\|\kappa_i\|^2=\min_j e^{Z_j}\ge e^{-\|Z\|_\infty}=e^{-|x|}.
\]
Writing $\kappa_i$ as a linear combination $\kappa_i= a_1\lambda_1+\ldots + a_d \lambda_d$, we obtain
\[
e^{-|x|} \le \langle\kappa_i, \kappa_i e^Z\rangle
\le \sum_{j=1}^d |a_j|\, |\langle\kappa_i, \lambda_j e^Z\rangle|.
\]
Note that the $a_i$ are uniformly bounded since $k$ and $l$ are contained in a compact set. Hence there exists $c>0$ and $j\in\{1,\ldots, d\}$ such that 
\[
|x_{ij}|=|\langle\kappa_i, \lambda_j e^Z\rangle | \ge c e^{-|x|},
\]
as claimed.

\medskip

Continuing with the proof of the proposition, suppose that there exists  $y\in h^{-1}gE_t\cap E_t$. We can assume without loss that $h^{-1} g= e^Y$ for some suitable $Y=(Y_1,\ldots, Y_n)\in \overline{\af^+}$. 

Since $y\in E_t$, the first part of the claim shows that $|y_{ij}|\le e^t$ for all $i,j=1,\ldots, d$. Moreover, since $y\in e^YE_t$, we can write $y=e^Y x $ for some $x\in E_t$. Then $y_{ij} = e^{Y_i} x_{ij}$ and the second part of the claim shows that for every $i$ there is $j_i$ such that $|x_{i j_i}|\ge ce^{-t}$. Putting this together, we obtain $e^t \ge c e^{Y_i} e^{-t}$, so that $e^{Y_i}\le c^{-1} e^{2t}$.

By similar arguments, we also know that $|x_{ij}|\le e^t$ for all $i,j$, and that for every $i$ there exists $j_i'$ with $e^{Y_i}|x_{i j_i'}|=|y_{i j_i'}|\ge ce^{-t}$, that is, $e^Y_i \ge c e^{-2t}$. 

These two inequalities together imply the first assertion of the corollary, and hence also the assertion the support of $A(\tau)$. The last assertion, on the sup of $A(\tau)$, is a direct consequence of the definition \eqref{def-Fbar}.
\end{proof}

Recall from \S\ref{Et-remarks} (iii) that $E_t\subset B_{t\|X^0\|_2}$. It then follows from Proposition \ref{cor:A-support} that we can take $R=\|X^0\| (2\tau+b)$ in Lemma \ref{lem:thick-thin}. Inserting this and the sup norm bound of Proposition \ref{cor:A-support} into Lemma \ref{lem:thick-thin} we obtain the second term in Theorem \ref{thm:GeomSide}.

\subsection{First term}

To establish  Theorem \ref{thm:GeomSide} it remains to estimate the first term in Lemma \ref{lem:thick-thin}, for the operator $\mathbf{A}(\tau)$ with kernel $A(\tau)$.

For a measurable set $E\subseteq G$ we write $\rho_{\Gamma\backslash G}(E)$ for the action by convolution of the normalized characteristic function of $E$ on $L^2(\Gamma\backslash G)$. Thus
\[
\rho_{\Gamma\backslash G}(E)f(x)=\frac{1}{m_G(E)}\int_E f(xg)dg.
\]
We shall in particular be interested in this operator for sets of the form $gE_t\cap E_t$.

We begin with the following elementary upper bound on the quantity
\[
\mathscr{A}(\tau)=\int_{\Gamma\backslash G} \int_G|A(\tau)(g,h)|^2 dg dh.
\]

\begin{lemma}\label{lem:L2:AX}
For all $\tau> 0$ we have
\[
\mathscr{A}(\tau)\leq \frac{1}{\tau^2}\int_{E_{2\tau+b}}\left(\int_{\max\{0,(|g|-b)/2\}}^\tau \frac{m_G(gE_t\cap E_t)}{m_G(E_t)}\| \rho_{\Gamma\backslash G}(gE_t\cap E_t)a\|_{L^2(\Gamma\backslash G)}dt\right)^2dg,
\]
where $b>0$ is as in Proposition \ref{cor:A-support}. 
\end{lemma}
\begin{proof}
From the definition \eqref{def-Fbar} of $A(\tau)$, we see that
\[
\mathscr{A}(\tau)=\frac{1}{\tau^2}\int_{\Gamma\backslash G}\int_G\bigg|\int_0^\tau \frac{1}{m_G(E_t)}\int_{gE_t\cap hE_t}a(x)dx dt\bigg|^2 dgdh.
\]
Changing variables $x\mapsto hx$ and $g\mapsto h^{-1}g$, this is
\[
\frac{1}{\tau^2}\int_{\Gamma\backslash G}\int_G\bigg|\int_0^\tau \frac{1}{m_G(E_t)}\int_{gE_t\cap E_t}A(hx,hx)dx dt\bigg|^2 dg dh.
\]
In view of Proposition \ref{cor:A-support}, the integral over $g\in G$ may be truncated at $|g|\leq 2\tau+b$ and the lower range of the $t$ integral may be truncated at $\max\{0,|g|-b)/2\}$. Furthermore,
\begin{align*}
\frac{1}{m_G(E_t)}\int_{gE_t\cap E_t}a(hx)dx&=\frac{m_G(gE_t\cap E_t)}{m_G(E_t)}\frac{1}{m_G(gE_t\cap E_t)}\int_{gE_t\cap E_t}a(hx)dx\\
&=\frac{m_G(gE_t\cap E_t)}{m_G(E_t)}\left(\rho_{\Gamma\backslash G}(gE_t\cap E_t)a\right)(h).
\end{align*}
We conclude the proof by an application of the Minkowski integral inequality.\end{proof}

We briefly return to the level of generality of Section \ref{sec:Weyl}, and let $G$ denote a connected non-compact simple Lie group with finite center. Let $(\pi,V_\pi)$ be a unitary representation of $G$. We recall that the \textit{integrability exponent} $q(\pi)$ of $\pi$ is defined as
\[
q(\pi)=\inf\{ q>0: \langle \pi(g)v_1,v_2\rangle\in L^q(G) \textrm{ for } v_1,v_2\textrm{ in a dense subspace of } V_\pi\}.
\]
A classical result of Borel--Wallach \cite{BW}, Cowling \cite{Cowling}, and Howe--Moore \cite{HM} states that $\pi$ has a spectral gap precisely when $q(\pi)<\infty$; we can effectively take this as our \textit{definition} of spectral gap in our setting. 

Nevo \cite{Nevo} (see also \cite{GV}*{Theorem 4.1}) has proved a mean ergodic theorem for measure preserving actions $G$ on a probability space $(X,\mu)$ whose associated unitary representation $L^2_0(X,\mu)$ has a spectral gap. We shall be interested in the action of $G$ on $X=\Gamma\bs G$ by right-translation, where  $\Gamma<G$ is a lattice. The associated unitary representation is then $\rho_{\Gamma\backslash G}^0$, the restriction of the right-regular representation $\rho_{\Gamma\backslash G}$ to the subspace $L^2_0(\Gamma\backslash G)$ of $L^2(\Gamma\backslash G)$ consisting of functions $f$ with $\int_{\Gamma\backslash G} f=0$. In this setting, they show the following result.

\begin{proposition}[Nevo]\label{prop:nevo}
Let $G$ be a connected non-compact simple Lie group with finite center. Let $\Gamma<G$ be a lattice. Then exist constants $\theta,C>0$, depending on the integrability exponents of $\rho_{\Gamma\backslash G}^0$ such that for any measurable $E\subseteq G$ of finite measure, we have 
\[
 \|\rho_{\Gamma\backslash G}^0(E)\|_2 \leq C m_G(E)^{-\theta}.
\]
\end{proposition}

A famous result of Kazhdan \cite{Kazhdan} states that, when $G$ furthermore has rank at least 2,
\[
\sup_{\pi: \, \pi^G=0}  q(\pi)<0,
\]
the supremum running over all unitary representations of $G$ having no $G$-invariant vectors. In this case, the constants of Proposition \ref{prop:nevo} can be taken independently of $\Gamma$. In particular, this is true of $G=\SL_d(\R)$ for $d\geq 3$.

We thus return to $G=\SL_d(\R)$, for $d\geq 3$, and again require $\Gamma<\SL_d(\R)$ to be a uniform lattice. Recall that $a\in L^2_0(\Gamma\backslash G)$. We can therefore apply Proposition \ref{prop:nevo} to estimate the quantity
\[
\|\rho_{\Gamma\backslash G}(gE_t\cap E_t)a\|_{L^2(\Gamma\backslash G)}= \|\rho_{\Gamma\backslash G}^0(gE_t\cap E_t)a\|_{L^2_0(\Gamma\backslash G)}
\]
appearing in Lemma \ref{lem:L2:AX}. Hence there is $\theta>0$, depending only on $d$, such that
\[
\mathscr{A}(\tau)
 \ll
 \frac{\|a\|_{L^2(\Gamma\backslash G)}^2}{\tau^2}\int_{E_{2\tau+b}}\left(\int_{\max\{0,(|g|-b)/2\}}^\tau \frac{m_G(gE_t\cap E_t)^{1-\theta}}{m_G(E_t)} dt\right)^2dg.
\]
To conclude the proof of Theorem \ref{thm:GeomSide}, it suffices to show the following estimate.

\begin{proposition}\label{prop-2int}
Let $g\in G$. Then for $\tau\gg 1$ we have
\[
\int_{E_{2\tau+b}}\left(\int_{\max\{0,(|g|-b)/2\}}^\tau \frac{m_G(gE_t\cap E_t)^{1-\theta}}{m_G(E_t)} dt\right)^2dg\ll \tau,
\]
the implied constant depending on $d$.
\end{proposition}

The main point in the proof of the above result is an estimate on the intersections volumes $m_G(gE_t\cap E_t)$, proved in the next paragraph.

\subsection{Intersection volumes}\label{subsec:intersection}

We seek to prove the following estimate on the intersection volumes $m_G(e^YE_t\cap E_t)$. See \S\ref{Et-remarks} for a discussion of some geometric aspects. 

\begin{proposition}\label{prop:upper:bound:intersection}
Let $Y\in\af^+$ and $t>0$. Then
\begin{equation}\label{reduce1}
m_G(e^YE_t\cap E_t )\ll e^{-\langle \rho, Y\rangle}m_G(E_t).
\end{equation}
\end{proposition}

\begin{proof}
Recall the definition of the Frobenius norm $\|\cdot \|$ from \S\ref{Et-remarks} (iv) and the Frobenius ball $\bm{E}_t$ from \eqref{def:Frob-ball}. We claim that
\begin{equation}\label{inclusion-claim}
E_t\subseteq \bm{E}_{t+\log \sqrt{d}}\subseteq E_{t+\log \sqrt{d}}.
\end{equation}
Note that $\|\cdot \|$ is bi-$K$-invariant, so that if $g=k_1 e^X k_2$, where $X=(X_1,\ldots ,X_d)\in\af$, then $\|g\|^2=e^{2X_1}+\cdots +e^{2X_d}$. Thus we may write $\bm{E}_t=K\exp (\bm{P}_t) K$, where
\[
\bm{P}_t=\{X=(X_1,\ldots ,X_d)\in\af: \max\{e^{2X_1}+\cdots +e^{2X_d}, e^{-2X_1}+\cdots +e^{-2X_d}\}\leq e^{2t}\}.
\]
We therefore wish to establish \eqref{inclusion-claim} for the sets $\bm{P}_t$ and $\cP_t=t\cP$, where $\cP$ is defined in \eqref{def:cP}: 
\begin{enumerate}
\item[$\triangleright$] If $|X_i|\leq t$ for all $i$ we have $\max\{e^{2X_1}+\cdots +e^{2X_d}, e^{-2X_1}+\cdots +e^{-2X_d}\}\leq de^{2t}$, so that $\cP_t\subseteq \bm{P}_{t+\log\sqrt{d}}$.
\item[$\triangleleft$] In the other direction, if $\max\{e^{2X_1}+\cdots +e^{2X_d}, e^{-2X_1}+\cdots +e^{-2X_d}\}\leq e^{2t}$ then $|X_i|\leq t$ for all $i$, so that $\bm{P}_t\subseteq \cP_t$.
\end{enumerate}

It is now enough to prove \eqref{reduce1} with $E_t$ replaced by $\bm{E}_t$. Indeed, suppose that
\begin{equation}\label{reduce2}
m_G(e^Y\bm{E}_t\cap \bm{E}_t)\ll e^{-\langle\rho,Y\rangle}m_G(\bm{E}_t)
\end{equation}
for any $Y\in\af$. We use \eqref{inclusion-claim} to bound $m_G(e^YE_t\cap E_t)$ by $m_G(e^Y\bm{E}_{t+\log\sqrt{d}}\cap \bm{E}_{t+\log\sqrt{d}})$. Inserting the estimate \eqref{reduce2} and the inclusions \eqref{inclusion-claim} again we find
\[
m_G(e^Y\bm{E}_{t+\log\sqrt{d}}\cap \bm{E}_{t+\log\sqrt{d}})
\ll e^{-\langle\rho,Y\rangle}m_G(\bm{E}_{t+\log\sqrt{d}})\leq e^{-\langle\rho,Y\rangle}m_G(E_{t+\log\sqrt{d}}).
\]
Since $m_G(E_{t+\log\sqrt{d}})\asymp m_G(E_t)$, by Corollary \ref{rem:volume}, we have reduced \eqref{reduce1} to \eqref{reduce2}.

We prove \eqref{reduce2} using the Iwasawa decomposition $G=ANK$ from \S\ref{sec:gen-notation}. Let $g=auk$ so that $\|g\|=\|au\|$ and $\|g^{-1}\|=\|u^{-1}a^{-1}\|$. From the triangle inequality of $\|\cdot \|$ on $M_d(\R)$, we see that the conditions $\|a\|\leq e^t$, $\|au-a\|\leq e^t$ together imply $\|g\|\leq e^t$. Similarly, the conditions $\|a^{-1}\|\leq e^t$, $\|u^{-1}a^{-1}-a^{-1}\|\leq e^t$ together imply $\|g^{-1}\|\leq e^t$. Thus we obtain the inclusion
\[
\bm{E}_t\subseteq \{g=auk\in G: \|a\|, \|a^{-1}\|, \|au-a\|, \|u^{-1}a^{-1}-a^{-1}\| \leq e^t\}.
\]
Applying the same argument to $e^Yg$, for $Y\in\af$, and dropping the last two conditions,
\[
e^Y\bm{E}_t\subseteq \{g=auk\in G:\|e^{-Y}a\|, \|e^Ya^{-1}\| \leq e^t\}.
\]

We conclude that $m_G(e^Y\bm{E}_t\cap \bm{E}_t)\leq I(Y,t)$, where
\[
I(Y,t)= \int\limits_{\substack{\|e^{\pm A}\|\leq e^t\\ \|e^{\pm (A-Y)}\|\leq e^t}} \int\limits_{\substack{\|e^Au-e^A\|\leq e^t\\ \|u^{-1}e^{-A}-e^{-A}\|\leq e^t}}du dA.
\]
Changing variables $u\mapsto e^Au-e^A+I$ in the $U$-integration, we obtain
\[
I(Y,t)= I_U(t) \int\limits_{\substack{\|e^{\pm A}\|\leq e^t\\ \|e^{\pm (A-Y)}\|\leq e^t}} e^{-\langle\rho,A\rangle}dA,
\]
where $I_U(t)=m_U(u\in U: \|u-I\|, \|u^{-1}-I\|\leq e^t)$. Changing variables $A\mapsto A-Y$ in the $A$-integration and dropping a condition, we obtain
\begin{align*}
\int\limits_{\substack{\| e^{\pm A}\|\leq e^t\\ \|e^{\pm (A-Y)}\|\leq e^t}} e^{-\langle\rho,A\rangle}dA&= e^{-\langle \rho,Y\rangle} \int\limits_{\substack{\|e^{\pm (A+Y)}\|\leq e^t\\ \| e^{\pm A}\|\leq e^t}} e^{-\langle\rho,A\rangle}dA\\
&\leq e^{-\langle \rho,Y\rangle}\int\limits_{\|e^{\pm A}\|\leq e^t} e^{-\langle\rho,A\rangle}dA.
\end{align*}
We write this last integral as $I_A(t)$. Using \eqref{Iwasawa-measure}, this establishes 
\begin{equation}\label{reverse}
m_G(e^Y\bm{E}_t\cap \bm{E}_t)\leq c_I e^{-\langle \rho,Y\rangle}I_A(t)I_U(t).
\end{equation}

We now go in reverse. For $A\in \af$ we change variables $u\mapsto e^{-A}u-e^{-A}+I$ to obtain
\[
e^{-\langle \rho,A\rangle}I_U(t)=m_U(u\in U: \|e^Au-e^A\|, \|u^{-1}e^{-A}-e^{-A}\|\leq e^t),
\]
so that
\[
I_A(t)I_U(t)=\int\limits_{\|e^{\pm A}\|\leq e^t} \int\limits_{\|(e^Au)^{\pm 1}-e^{\pm A}\|\leq e^t} du dA.
\]
Again by the triangle inequality and \eqref{Iwasawa-measure} this gives
\[
I_A(t)I_U(t)\leq \iint\limits_{\|(au)^{\pm 1}\|\leq 2e^t} da du=c_I^{-1}m_G(\bm{E}_{t+\log 2})\asymp m_G(\bm{E}_t).
\]
Combining this last inequality with \eqref{reverse} yields \eqref{reduce2}. \end{proof}

\subsection{Proof of Proposition \ref{prop-2int}} 
From Proposition \ref{prop:upper:bound:intersection} and Lemma \ref{rem:volume} we have, for $Y\in\af^+$
\[
\frac{m_G(e^YE_t\cap E_t)^{1-\theta}}{m_G(E_t)}\ll e^{-(1-\theta)\langle \rho, Y\rangle}m_G(E_t)^{-\theta} \ll e^{-(1-\theta)\langle \rho, Y\rangle}e^{-t\theta \langle 2\rho, X^0\rangle},
\]
so that, using \eqref{Et-polar},
\begin{align*}
\int_{E_{2\tau+b}}&\left(\int_{\max\{0,|g|-b)/2\}}^\tau \frac{m_G(gE_t\cap E_t)^{1-\theta}}{m_G(E_t)} dt\right)^2dg\\
&\ll\int_{\mathscr{P}_{2\tau+b}^+}\left(\int_{\max\{0,(\|Y\|_\infty-b)/2\}}^\tau e^{-t\theta\langle 2\rho,X^0\rangle} dt\right)^2e^{-(1-\theta)\langle 2\rho, Y\rangle}J(Y)dY.
\end{align*}
Inserting $J(X)\ll e^{2\langle \rho,X\rangle}$ for $X\in\af^+$ from \eqref{J(H)}, and evaluating the $t$-integral, we majorize the above expression by
\[
\int_{\mathscr{P}_{2\tau+b}^+} e^{-\max\{0,\|Y\|_\infty-b\}\theta\langle 2\rho,X^0\rangle} e^{\theta\langle 2\rho, Y\rangle}dY.
\]
We break up the last integral as $I_1+I_2$, according to $\af_1^+=\mathscr{P}_{b}^+$ and
\begin{align*}
\af_2^+(\tau)&=\{Y\in\af^+: b\leq \|Y\|_\infty\leq 2\tau+b\}.
\end{align*}
Since $\af^+_1$ is independent of $\tau$ (as is the integrand), we have $I_1\ll 1$. Next, we have
\[
I_2\ll \int_{\af_2^+(\tau)} e^{2\theta\left(\langle \rho,Y\rangle -\|Y\|_\infty\langle \rho,X^0\rangle\right)}dY.
\]
We will drop several of the conditions on $Y$ in the course of the proof to simplify the integral.

 Write $\tau' = 2\tau +b$. We distinguish the case of $d$ being even and odd. Define
\begin{equation}\label{defn-s}
 s =\begin{cases}
       d/2		&\text{if $d$ even},\\
       (d+1)/2		&\text{if $d$ odd}.
      \end{cases}
\end{equation}
 
 We first assume that $d\geq 3$ is odd. We can write $Y$ in the integral as
\[
Y= (a_1,\ldots, a_s, -a_1-\cdots -a_s + b_1+\cdots+ b_s ,-b_s,\ldots, -b_1)
\]
with $\tau'\ge a_1\ge\cdots \ge a_s\ge 0$ and $\tau'\ge b_1\ge\cdots\ge b_s\ge0$. Then $\langle \rho, Y\rangle = \sum_{i=1}^s (s-i) (a_i+b_i)$, and, using Lemma \ref{lem:X0:explicit}, $\langle \rho, X^0 \rangle = 2 A$ with $A=\sum_{i=1}^s(s-i)$. Moreover, $\|Y\|_\infty = \max\{a_1, b_1\}$. Hence
\[
I_2(\tau)\ll \int\ldots\int  e^{2\theta\left( \sum_{i=1}^s (s-i) (a_i+b_i)- 2A\max\{a_1,b_1\}\right)},
\]
where the integral runs over
\[
\left\{
a_1,\ldots, a_s, b_1,\ldots, b_s\bigg\vert 
\begin{aligned}
 \tau'\ge a_1\ge0,\; a_1\ge a_2\ge0,\; & \ldots,\; a_{s-1}\ge a_s\ge0,\\
\tau'\ge b_1\ge0,\; b_1\ge b_2\ge0,\; & \ldots,\; b_{s-1}\ge b_s\ge0.
\end{aligned}
\right\}
\]
We now integrate first $a_s$ and $b_s$ and continue successively with $a_{s-1},\ldots, a_2$ and $b_{s-1},\ldots, b_2$ until only $a_1$ and $b_1$ remain to be integrated. In this way we obtain
\[
I_2(\tau)\ll\iint_0^{\tau'} e^{2\theta A(a_1+b_1 - 2\max\{a_1, b_1\})}\, da_1\, db_1
= \iint_0^{\tau'} e^{- 2\theta A|a_1-b_1|}\, da_1\, db_1
= 2\tau' \int_0^{\tau'} e^{- 2\theta Ax}\, dx
\ll \tau.
\]

If $d$ is even, then $d\ge4$ and $s\ge2$. We can write $Y= (a_1,\ldots, a_s, -b_s, \ldots, -b_1)$ with $\tau'\ge a_1\ge\cdots\ge a_s\ge 0$, $\tau'\ge b_1\ge\cdots\ge b_s\ge 0$, and 
\begin{equation}\label{eq:dependence:coeff}
a_1+\cdots + a_s = b_1+\cdots + b_s. 
\end{equation}
Then $\|Y\|_\infty = \max\{a_1, b_1\}$ and
\[
\langle \rho, Y\rangle 
= \sum_{i=1}^s (s-i+1/2)(a_i+b_i)= \sum_{i=1}^s (s-i+1)a_i + \sum_{i=1}^{s-1} (s-i)b_i,
\]
where we have used \eqref{eq:dependence:coeff} to replace $b_s$. Furthermore, using Lemma \ref{lem:X0:explicit}, we have $\langle\rho, X^0\rangle =A+ B$, where we have put
\[
A=\sum_{i=1}^s (s-i+1)\quad\text{ and }\quad B= \sum_{i=1}^s (s-i).
\] 
Extending the domain of integration if necessary, we obtain
\[
I_2(\tau)\leq \int\ldots\int e^{2\theta\left(\sum_{i=1}^s (s-i+1)a_i + \sum_{i=1}^{s-1} (s-i)b_i - (A+B)\max\{a_1,b_1\}\right)},
\]
where the integral runs over 
\[
\left\{ a_1,\ldots, a_s, b_1,\ldots, b_{s-1}\bigg\vert
\begin{aligned}
\tau'\ge a_1\ge0,\; a_1\ge a_2\ge0,\; & \ldots,\; a_{s-1} \ge a_s\ge0,\\
\tau'\ge b_1\ge0,\; b_1\ge b_2\ge0,\; & \ldots,\; b_{s-2}\ge b_{s-1}\ge0
\end{aligned}
\right\}.
\]
As before, we successively integrate over $a_s,a_{s-1},\ldots, a_2$ and $b_{s-1}, b_{s-2},\ldots, b_2$, obtaining
\[
I_2(\tau)\ll\iint_{0}^{\tau'} e^{2\theta ( Aa_1 + Bb_1 - (A+B)\max\{a_1, b_1\})}\, da_1\, db_1
\le  \iint_{0}^{\tau'} e^{-2\theta B |a_1-b_1|} \, da_1\, db_1
\ll \tau
\]
as desired.\qed

\begin{appendix}
\section{Cones and volumes}\label{appendix:X0}
Recall the definition of $\cP^+$ and $E_t$ in \S\ref{sec:Et}. 
Our goal is to write $\cP^+$ as the intersection of $\overline{\af^+}$ with a suitable cone in $\af$ and to calculate the asymptotic volume of the set $E_t$.

\subsection{Cones}
We will identify $\af$ as before with the subspace of $\R^d$ consisting of all vectors $X=(X_1,\ldots, X_d)$ with $X_1+\cdots+X_d=0$. We also identify $\af$ with the set of all trace-zero diagonal matrices whenever convenient. 

\begin{lemma}\label{lem:X0:explicit}
Let $X^0\in \cP$ be any point in $\cP$ satisfying
\[
\langle X^0, \rho\rangle = \max_{X\in\cP^+} \langle X, \rho\rangle.
\]
Let $\{e_1,\ldots, e_d\}$ denote the usual standard basis of $\R^d$.
Then
\begin{equation}\label{defX0}
X^0 = \begin{cases}
							e_1+\cdots + e_s - e_{s+1}-\cdots - e_d		&\text{if }d\text{ even};\\
							e_1+\cdots + e_{s-1} - e_{s+1}-\cdots -e_d		&\text{if }d\text{ odd},
\end{cases}
\end{equation}
where $s$ is defined in \eqref{defn-s}. In particular, $X^0$ is unique.
\end{lemma}

\begin{proof}
We write the half-sum of positive roots $\rho\in \af^*$ in coordinates as $\rho = (\rho_1,\ldots, \rho_d)\in \R^d\simeq (\R^d)^*$, where $\rho_1+\cdots+\rho_d=0$. We have $\rho_1\ge\cdots\ge\rho_d$ and $\rho_i=-\rho_{d+1-i}$ for every $i\le d/2$ so that if $d$ is odd, then $\rho_{(d+1)/2}=0$. Let $X=(X_1,\ldots, X_d)\in\cP$. Then
\[
 \langle X, \rho\rangle = \sum_{i\le d/2} (X_i - X_{d+1-i}) \rho_i.
\]
Since $X\in\cP$, we have $-2\le X_i-X_{d+1-i}\le 2$. Furthermore, $\rho_i>0$ for $i\le d/2$, so this sum is maximized for $X_i=1$ and $X_{d+1-i}=-1$. This proves our assertion for $d$ even. For $d$ odd, we finally note that $X_1+\cdots+X_d=0$ forces $X_{(d+1)/2}=0$.
\end{proof}

\begin{lemma}\label{mu-basis}
There is a basis $\mu_1,\ldots ,\mu_{d-1}\in\af^*$ such that $\mu_i(X^0)=1$, $i=1,\ldots, d-1$, and  $\cP^+=\CmC\cap\overline{\af^+}$, where
\[
\CmC=X^0+\CmC^0 \quad\text{and}\quad \CmC^0=\{X\in\af: \mu_i(X)\leq 0\}.
\]
In particular, if $\{\beta_i^\vee\}$ is the basis in $\af$ which is dual to $\{\mu_i\}$ then
\[
\CmC^0=\left\{\sum_{i=1}^{d-1} x_i \beta_i^{\vee}: \forall i\; x_i\leq 0 \right\}
\]
and $ X^0= \beta_1^\vee +\cdots + \beta_{d-1}^\vee$.
\end{lemma}\label{eq:X0:sum:betas}

\begin{proof}
Note that 
\begin{equation}\label{inter}
\cP^+=\{X\in\af: X_1\leq 1,\; -X_d\leq 1\}\cap \overline{\af^+}.
\end{equation}
We rewrite \eqref{inter} using the system of fundamental weights $\varpi_i\in\af^*$:
\begin{equation}\label{inter2}
\cP^+=\{X\in\af: \varpi_1(X)\leq 1, \varpi_{d-1}(X)\leq 1\}\cap \overline{\af^+}.
\end{equation}
Since $\varpi_1(X^0)=\varpi_{d-1}(X^0)=1$, we can put $\mu_1=\varpi_1$ and $\mu_{d-1}=\varpi_{d-1}$. The strategy is then to complete $\{\mu_1,\mu_{d-1}\}$ to a basis of linear forms $\{\mu_i\}$ in such a way that
\begin{enumerate}
\item\label{1} $\mu_i(X^0)=1$ for all $i=1,\ldots d-1$;
\item\label{2} the conditions $X\in\overline{\af^+}$ and $\mu_1(X),\mu_{d-1}(X)\leq 1$ imply $\mu_i(X)\leq 1$ for $i=2,\ldots,d-2$.
\end{enumerate}
In this case, property \eqref{1} shows that
\[
\CmC =  X^0+\CmC^0=\{X\in\af\mid \forall i:\, \mu_i(X - X^0)\le 0\}=\{X\in\af\mid \forall i:\, \mu_i(X)\le 1\},
\]
and then property \eqref{2} combines with \eqref{inter2} to establish $\cP^+=\CmC \cap\overline{\af^+}$.

Note that, relative to the standard basis $\{e_i\}$ of $\R^d$, $\mu_1$ is the (restriction to $\af$ of the) first coordinate functional and $\mu_{d-1}$ is minus the (restriction to $\af$ of the) last coordinate functional. Thus if the $\pm\mu_i$, for $i=2,\ldots, d-2$, are a linearly independent subset of the remaining $d-2$ coordinate functionals, then \eqref{2} follows from the fact that $1\geq X_1\geq\cdots \geq X_d\geq -1$ implies $\pm X_i\leq 1$. To assure property (1) we simply need to choose the sign suitably and avoid the $(d+1)/2$ coordinate for odd $d$. Thus for $d$ odd we put
\[
\mu_i = \begin{cases}
\varpi_i-\varpi_{i-1}			&\text{if } 2\le i\le s-1;\\
\varpi_i-\varpi_{i+1}			&\text{if } s\le i \le d-2,\\
\end{cases}
\]
and for $d$ even we put
\[
\mu_i = \begin{cases}
\varpi_i-\varpi_{i-1}						&\text{if } 2\leq i\le s;\\
\varpi_{i}-\varpi_{i+1} 						&\text{if } s+1\leq i\leq d-2,
\end{cases}
\]
where $s$ is defined in \eqref{defn-s}. Using the definition \eqref{defX0} of $X^0$ we quickly verify property \eqref{1}. \end{proof}

\begin{remark}\label{rem:positive:rho}
 Using the explicit description of the basis $\{\mu_i\}_{1\le i\le d-1}$, one can easily see that writing $\rho$ as a linear combination of the $\mu_i$, all the coefficients are positive integers. Hence if $Y\in \cC^0$, then $\langle\rho, Y\rangle<0$. 
\end{remark}

\subsection{Volumes}

\begin{proposition}\label{Prop:volumes:appendix}
 Suppose $P\subseteq \overline{\af^+}$ is a convex polytope. Let $V_P$ denote the vertices of $P$, and suppose that $V_P\ni X\mapsto \langle\rho, X\rangle$ takes its maximum at exactly one vertex $v_0\in V_P$. Then there exist $c, \delta>0$ such that
 \begin{equation}\label{eq:exponential_polytope}
  \int_{P_t} \sum_{w\in W} \sigma(w) e^{\langle w\rho, X\rangle}\, dX
  = c e^{t\langle\rho, v_0\rangle}\left(1 + O(e^{-t\delta})\right)
 \end{equation}
as $t\rightarrow\infty$. Here $P_t=\{tX\mid X\in P\}$ and $\sigma(w)$ denotes the signature of $w\in W$.
\end{proposition}
Before starting the proof of this proposition, we note that 
\[
\sum_{w\in W} \sigma(w) e^{\langle w\rho, X\rangle}\ge 0
\]
for all $X\in\overline{\af^+}$ because of the convexity of the exponential function.
\begin{proof}
 First note that 
 $\langle w\rho, X\rangle\le \langle\rho, X\rangle$
 for all $X\in\overline{\af^+}$, and in fact 
 $\langle w\rho, X\rangle< \langle\rho, X\rangle$
  if $X\in\af^+$. If $P$ is not simple, we divide it into disjoint simple polytopes $P^1,\ldots, P^r$ whose respective sets of vertices we denote by $V^1,\ldots, V^r$. The vertices $V_P$ are contained in the union $\bigcup_j V^j$. By Brion's formula \cite{BV} there exist non-zero coefficients $a_v^j\in\R$, $j=1,\ldots, r$, $v\in V^j$, such that 
 \[
  \int_{P_t^j} e^{\langle w\rho, X\rangle}\, dX
   = \sum_{v\in V^j} a_v^je^{t\langle w\rho, v\rangle}.
 \]
We can find $\delta>0$ such that for every $j$ and $v\in V^j$,  $\langle w\rho, v\rangle\le \langle \rho, v_0\rangle-\delta$ unless $w^{-1}v = v_0$. Hence,
 \begin{equation}\label{eq:exp_polytope}
   \int_{P_t} \sum_{w\in W} \sigma(w) e^{\langle w\rho, X\rangle}\, dX
   = \left(\sum_{\substack{j,v, w:\\ w^{-1}v=v_0}} a_v^j \sigma(w)\right)   e^{t\langle \rho, v_0\rangle}.
 \end{equation}
It remains to argue that this last sum in brackets is a positive number. If $v_0$ is such that $\langle w\rho, v\rangle = \langle \rho, v_0\rangle$ if and only if $w=1$ and $v=v_0$, then the left hand side of \eqref{eq:exponential_polytope} tends to $+\infty$ as $t\rightarrow\infty$. Moreover, the sum on the right hand side of \eqref{eq:exp_polytope} has exactly one non-zero term in that situation which consequently must be positive. If $v_0$ does not satisfy this condition,  we choose another polytope $P'\subseteq P$ with regular vertices (i.e., $V_{P'}\subseteq \af^+$) that are sufficiently close to the original vertices of $P$. Then
\[
  \int_{P_t} \sum_{w\in W} \sigma(w) e^{\langle w\rho, X\rangle}\, dX
  \le  \int_{P_t'} \sum_{w\in W} \sigma(w) e^{\langle w\rho, X\rangle}\, dX,
\]
and the right hand side grows like $ce^{tc'}$ for suitable $c>0$ and $\langle\rho,v_0\rangle \ge c'>\langle\rho, v_0\rangle-\delta$ provided $P'$ is chosen sufficiently close to $P$. This proves that the sum in brackets on the right hand side of \eqref{eq:exp_polytope} is a positive number.
\end{proof}

We apply the above result to estimate the volume of $E_t$.

\begin{corollary}\label{rem:volume}
There exists $c, \delta>0$ such that we have
\[
  m_G(E_t)= 
  c e^{2t\langle \rho, X^0\rangle}\left(1+ O(e^{-t\delta})\right)
 \]
 as $t\rightarrow\infty$.
\end{corollary}
\begin{proof}
By Cartan decomposition we have 
\[
 m_G(E_t) = \int_{P_t} J(X)\, dX
 = 2^{-|\Phi^+|}\int_{P_t} \sum_{w\in W} \sigma(w) e^{2\langle w\rho, X\rangle} \, dX
\]
where $\sigma(w)$ denotes the signature of $w$. The assertion of the lemma then follows from Proposition \ref{Prop:volumes:appendix}.
\end{proof}

\section{Angles and inner products}\label{sec:cones}

The purpose of this appendix is to prove the following result, used in the course of the proof of Lemma \ref{lem:wInt:nonzero}.

Recall the notation of that lemma. Let $X^0$ be as in \eqref{defnX0} and let $M$ denote its centralizer in $G$. For $Y\in \af$ we write $Y= Y_M+Y^M$ for unique $Y_M\in \af_M$ and $Y^M\in \af^M$. As before, for $\eta>0$ we are writing $\cT_\eta$ for the set of all $\lambda\in \af^*_\C$ such that $\Re\lambda$ lies in the convex hull of the Weyl group orbit of $\eta\rho$. The cone $\cC^0$ is defined in Lemma \ref{mu-basis}.
\begin{proposition}\label{claims}
There are constants $C_1, C_2>0$, depending only on $d$, such that for every $Y\in \cC^0$, $\lambda\in \cT_\eta$, and $w\in W$  we have 
\begin{enumerate}[label=(\roman*)]
\item\label{claim2} $|\langle w \rho, Y_M\rangle|\leq   - C_1\langle \rho, Y\rangle$;
\item\label{claim1} $|\langle \Re\lambda,  Y_M\rangle| \leq - C_2\eta  \langle \rho, Y_M\rangle$.
\end{enumerate}
In particular, there is a constant $C>0$, depending only on $d$, such that for all  $Y\in \cC^0$ and $\lambda\in \cT_\eta$ we have
\[
 |\langle \Re \lambda, Y_M\rangle|\le - \eta C \langle\rho, Y\rangle.
\]
\end{proposition}

To prove the proposition we shall work abstractly with linear forms on euclidean space $\R^d$, equipped with its standard inner product. Namely, we make the usual identification of the standard Cartan subalgebra $\af_0$ of $\mathfrak{gl}_d(\R)$ with $\R^d$, so that $\af=\af_0^G$ identifies with the trace-zero hyperplane
\[
\cH=\left\{X=(X_1,\ldots, X_d)\in\R^d : X_1+\cdots + X_d = 0\right\}.
\]
We are interested in the linear form $L(Y) = \langle\rho, Y\rangle$ on $\R^d$ or $\cH$. 
It will be convenient to define a cone $\cC'$ in $\R^d$ such that $\cC'\cap\cH$ coincides with the cone $-\cC^0$ of Proposition \ref{mu-basis}. Property \ref{claim2} will then follow from a similar maximizing property of $L$ on $\cC'$. 
Property \ref{claim1} requires bounds on the angles the vectors in $-\cC^0$ can form with $X^0$, which we deduce from the relative position of $X^0$ with $\mathscr{H}$.

\subsection{Positive linear forms}
Let $\{e_1,\ldots, e_d\}$ denote the standard basis of $\R^d$.

\begin{definition}\label{defn:C'}
If $d$ is even, let $\cC'\subset\R^d$ be the closed orthant 
\[
\cC'={\rm cone}_{\R_{\geq 0}}\{e_1,\ldots, e_{d/2}, -e_{d/2+1},\ldots, -e_d\}.
\]
If $d$ is odd, let $\cC'=\cO^+\cup \cO^-\subset\R^d$ be the union of the closed orthants
\begin{align*}
\cO^+&={\rm cone}_{\R_{\geq 0}}\{e_1,\ldots, e_{(d+1)/2-1},  e_{(d+1)/2}, -e_{(d+1)/2+1},\ldots, -e_d\}\\
 \cO^-&={\rm cone}_{\R_{\geq 0}}\{e_1,\ldots, e_{(d+1)/2-1}, -e_{(d+1)/2}, -e_{(d+1)/2+1}, \ldots, -e_d\}.
 \end{align*}
 \end{definition}

\begin{remark}
Recall the explicit description of $X^0\in\R^d$ as given in \eqref{defX0} of Appendix \ref{appendix:X0}. If $V_d$ denotes the set of vertices of the unit cube $[-1,1]^d$, then $\pm X^0\in V_d$ for $d$ even and $\pm X^0\pm e_{(d+1)/2}\in V_d$ for $d$ odd. In particular, when $d$ is odd $-X^0$ lies on the  edge of $[-1,1]^d$ connecting the vertices $X^0\pm e_{(d+1)/2}$. We deduce that the origin of $\R^d$ is a vertex of $-X^0+[-1,1]^d$ if $d$ is even, and it is the midpoint of an edge of $-X^0+[-1,1]^d$ if $d$ is odd.

It is easy to see that when $d$ is even, $-\cC'$ is the unique orthant in $\R^d$ containing $-X^0+[-1,1]^d$, and when $d$ is odd, $-X^0+[-1,1]^d\subset -\cC'$.
\end{remark}

\begin{remark}\label{rem:lin-subspace}
If $d$ is even, then $\cC'$ does not contain any non-trivial linear subspace, and if $d$ is odd, the only non-trivial linear subspace of $\cC'$ is the one-dimensional space $\R e_{(d+1)/2}$.
\end{remark}

\begin{definition}\label{def:pos-lin-form}
We call a linear form $L:\R^d\longrightarrow \R$ {\rm positive} if it is non-negative on $\cC'$ and positive on $\cC'\smallsetminus \{0\}$ if $d$ is even, and positive on $\cC'\smallsetminus \R e_{(d+1)/2}$ if $d$ is odd. 
\end{definition}

Let $\cL\subseteq \R^d$ denote the ray $\R_{\ge0} X^0$. Since $X^0\in\cC'$, and $\cC'$ is a cone, we have $\cL\subset\cC'$.

We identify the space of linear forms on $\R^d$ with $(\R^d)^*\simeq \R^d$ via the dual basis to $\{e_1,\ldots, e_d\}$.  
If $L$ is a linear form, written as $L=(L_1,\ldots, L_d)$ with respect to this identification, and $\sigma\in S_d$ is a permutation, we define the linear form $\sigma L$ by the rule $\sigma L:=(L_{\sigma(1)},\ldots, L_{\sigma(d)})$. We identify 

\begin{lemma}\label{lem:L-max}
Let $L\in(\R^d)^*$ be a positive linear form. Let $L'$ be a form in the closed convex hull of $\{\sigma L:\sigma\in S_d\}\subseteq (\R^d)^*$. Then for all $X\in\cL$ we have $L'(X)\le L(X)$.
\end{lemma}

\begin{proof}
A linear form $L=(L_1,\ldots, L_d)$ is positive if and only if 
\[
\begin{cases}
L_1,\ldots, L_{d/2}>0, \;\; L_{d/2+1},\ldots, L_d<0, &\text{if }d \text{ is even};\\
L_1,\ldots, L_{(d+1)/2-1}>0, \;\; L_{(d+1)/2}=0,\;\; L_{(d+1)/2+1},\ldots, L_d<0,					&\text{if }d \text{ is odd}.
\end{cases}
\]
It follows that for such $L$ we have $L(X^0) = |L_1|+\cdots + |L_d|$, hence $\sigma L(X^0)\le L(X^0)$ for every $\sigma\in S_d$. Hence if $L'$ is in the closed convex hull of $\{\sigma L:\sigma\in S_d\}$ in $(\R^d)^*$, we also have $L'(X^0)\le L(X^0)$.
\end{proof}

\subsection{Hyperplane intersections}
We now put $\cC'' = \cC'\cap \cH$. 

\begin{lemma}\label{lem:acute:angle}
 $\cC''$ is a convex cone in $\R^d$ of dimension $d-1$ that satisfies $\cC''\cap (-\cC'')=\{0\}$. Moreover, there exists $\gamma>0$ such that the angle between any non-zero vector in $\cC''$ and the vector $X^0$ is at most $\pi/2-\gamma$.  
\end{lemma}
\begin{proof}
The first part of the lemma follows directly from the description of $\cC'$ in Definition \ref{defn:C'}, coupled with Remark \ref{rem:lin-subspace}.

For the assertion on the bound of the angles, we proceed as follows: if $d$ is even, then $\cC'$ is just a single orthant in $\R^d$, hence the angle between any two vectors is bounded by $\pi/2$. Since $X^0$ does not lie on the boundary of $\cC'$, there exists $\gamma>0$ such that the angle between $X^0$ and any vector in $\cC'$ is bounded from above by $\pi/2-\gamma$.

For $d$ odd, recall that $\cC'=\cO^+\cup\cO^-$ and $X^0\in \cO^+\cap \cO^-$, hence the angle between $X^0$ and any vector in $\cC'$ is at most $\pi/2$. Moreover, it is readily seen from the explicit description of $X^0$ and $\cO^\pm$ that the only vectors in $\cC'$ that are orthogonal to $X^0$ are the vectors on the line $\R e_{(d+1)/2}$. Since $e_{(d+1)/2}\not\in\cH$ and $\cC''=\cC'\cap\cH$ is closed, the angle between $X^0$ and non-zero vectors in $\cC''$ must be bounded from above by $\pi/2-\gamma$ for some $\gamma>0$.
\end{proof} 

Note that $\cL\subset\cC''$ since $X^0\in\mathscr{H}$.

\begin{proposition}\label{prop:compactness:proj}
Let $\cM\subseteq \cH$ be a subvectorspace that is orthogonal to $\cL$. Then for every $X\in \cL$ the section
\[
(X+\cM)\cap\cC''
\]
is compact. Moreover, if $\cM$ is of dimension $d-2$, then
\[
\cC'' = \bigcup_{X\in\cL} \left((X+\cM)\cap \cC''\right),
\]
and writing $X= rX^0$ for a suitable $r\ge 0$ we have $(X+\cM)\cap \cC'' = r\left((X^0+\cM)\cap\cC''\right)$. 
\end{proposition}
\begin{proof}
Follows from Lemma \ref{lem:acute:angle}.
\end{proof}

\subsection{Application to linear forms}

In the following lemma, we shall take $L$ to be a positive linear form in the sense of Definition \ref{def:pos-lin-form}, and $\cM\subseteq \cH$ a vector space of dimension $d-2$ which is orthogonal to $\cL$. 

\begin{lemma}\label{lem:linear:forms:inequ}
There exists a constant $c>0$ such that $L(X+Z)\le c L(X)$ for all $X\in\cL$ and $Z\in\cM$ with $X+Z\in \cC''$.
\end{lemma}
\begin{proof}
For fixed $X\in\cL$, the set of all $Z\in \cM$ such that $X+Z\in\cC''$ is compact by Proposition \ref{prop:compactness:proj}. In particular there is $D\subseteq \cM$ compact such that $(X^0+ \cM)\cap \cC'' = X^0 + D$. Hence any $X+Z\in\cC''$ with $X\in\cL$ and $Z\in\cM$ can be written as $X+ Z= r X^0 + rZ'$ for some suitable $r\ge0$ and $Z'\in D$. Putting 
\[
c= 1+ \frac{\max_{Z'\in D} L(Z')}{L(X^0)}
\]
 we then get $L(X+ Z) = r L(X^0) + r L(Z') \le c L(X)$.\qedhere
\end{proof}

\subsection{Proof of Proposition \ref{claims}}
We now apply the previous results to our situation. Namely, we identify $\af_0$ with $\R^d$ as usual so that $\af=\af_0^G$ is the hyperplane $\cH$. The cone $\cC''= \cC'\cap\cH$ then coincides with the cone $-\cC^0$ that we considered in Appendix \ref{appendix:X0}, and $X^0$ coincides with the vector of the same name from Appendix \ref{appendix:X0}. 

We consider the linear form $L(Y) = \langle\rho, Y\rangle$, $Y\in\cC$. Let $\af_M\subseteq \af$ be as before. Then $\R X^0\subseteq \af_M$. In fact, if $d$ is even, then $\af_M=\R X^0$ and $\af^M$ is orthogonal to $\af_M$ of dimension $d-2$. If $d$ is odd, then $\af_M$ has dimension $2$. In that case, let $V\subseteq \af_M$ be the orthogonal complement to the line $\R  X^0$ so that $V\oplus \af^M$ is orthogonal to $\R X^0$ of dimension $d-2$. 

As before, if $Y\in\cC''$, we write $Y=Y_M + Y^M$ for unique $Y_M\in \af_M$ and $Y^M\in \af^M$. If $d$ is even, then $Y_M\in\R_{\ge0} X^0$, and if $d$ is odd, we can uniquely write $Y_M = Y_M' + Y_M''$ with unique  $Y_M'\in \R_{\ge 0} X^0$ and $Y_M''\in V$. 

Having set up the notation, we now note that $\langle\rho, \cdot\rangle$ vanishes on $V$, hence part \ref{claim2} of Proposition \ref{claims} is an immediate consequence of Lemma \ref{lem:L-max}.

For \ref{claim1} we first note that by definition of $\cT_\eta$, $\Re\lambda$ is contained in the convex hull of the Weyl group orbit of $\eta\rho$ so that $\langle\Re\lambda, wX^0\rangle \le \eta \langle \rho, X^0 \rangle$ by Lemma~\ref{lem:linear:forms:inequ} for every $w\in W$. This establishes \ref{claim1} for $d$ is even. For $d$ odd, we have $\eta\langle\rho, Y_M\rangle = \eta\langle\rho, Y_M'\rangle\ge \langle\Re\lambda, Y_M ' \rangle$ since $\langle\rho, \cdot \rangle$ vanishes on $V$. Taking $Y_M'=X^0$, the possible $Y_M''$ lie in a compact set so that $\langle\Re\lambda, Y_M'' \rangle$ is bounded from above by $\eta c$ for $c$ some absolute constant. Scaling $Y_M'=X^0$ by a scalar, the compact set of $Y_M''$ gets scaled by the same scalar. Hence there is $C>0$ such that for any $Y\in\cC''$ we have $\langle \Re\lambda,  Y_M\rangle \le \eta C \langle \rho, Y_M\rangle$\qed

\end{appendix}

 \end{document}